\documentclass[a4paper,10pt]{article}

\usepackage[ansinew]{inputenc}
\usepackage{amsmath}
\usepackage{amsfonts}
\usepackage{amssymb}
\usepackage{amsthm}
\usepackage{graphicx}
\usepackage{color}
\usepackage{multirow}
\usepackage{enumerate}
\usepackage[all]{xy}
\usepackage{url}
\usepackage{varioref}

\DeclareGraphicsExtensions{.pdf,.png}

\newtheoremstyle{mystyle}
  {}
  {}
  {}
  {}
  {\bfseries}
  {}
  {\newline}
  {}

\newtheorem{thm}{Theorem}[section]
\newtheorem{lem}[thm]{Lemma}
\newtheorem{cor}[thm]{Corollary}

\newtheorem*{thma}{Theorem}
\newtheorem{deff}[thm]{Definition}

\newcommand{\C}{\mathbb{C}}
\newcommand{\R}{\mathbb{R}}

\newcommand{\Z}{\mathbb{Z}}

\newcommand{\PD}{\P^2_{\C}}
\newcommand{\PDD}{\P^3_{\C}}

\newcommand{\ma}[1]{\left(\begin{array}{ll}
1 & 0 \\
0 & #1
\end{array}
\right)}

\DeclareMathOperator{\id}{Id}
\DeclareMathOperator{\h}{h}

\DeclareMathOperator{\Dim}{dim}

\DeclareMathOperator{\Dic}{Dic}
\DeclareMathOperator{\Tr}{Tr}

\DeclareMathOperator{\Sym}{Sym}
\DeclareMathOperator{\Aut}{Aut}

\DeclareMathOperator{\Pic}{Pic}
\DeclareMathOperator{\GL}{GL}
\DeclareMathOperator{\NE}{NE}
\DeclareMathOperator{\Biol}{Bihol}

\DeclareMathOperator{\Bl}{Bl}
\DeclareMathOperator{\Fix}{Fix}
\DeclareMathOperator{\e}{e}

\DeclareMathOperator{\diag}{diag}

\renewcommand{\P}{\mathbb{P}}
\renewcommand{\c}{c}

\renewcommand{\O}{\mathcal{O}}

\newif\ifVERSIONEPROLISSA 
\VERSIONEPROLISSAfalse 

\newif\ifCUTSECTION 
\CUTSECTIONfalse 

\title{Groups Acting Freely on Calabi-Yau Threefolds \\  Embedded in a Product of del Pezzo Surfaces}
\author{Gilberto Bini\footnote{Universit\`{a} degli Studi di Milano - Dipartimento di Matematica ``F. Enriques'' - Via C. Saldini, 50 - 20133 Milano (Italy). {\sc E-mail}: {\tt gilberto.bini@unimi.it}}$\,$ and Filippo F. Favale\footnote{Universit\`{a} degli Studi di Pavia (Italy) - Dipartimento di Matematica ``F. Casorati'' - Via Ferrata, 1 - 27100 Pavia. {\sc E-mail}: {\tt filippo.favale@unipv.it}} }

\begin{document}

\maketitle
\begin{abstract}
\noindent In this paper, we investigate quotients of Calabi-Yau manifolds $Y$ embedded in Fano varieties $X$ which are products of two del Pezzo surfaces - with respect to groups $G$ that act freely on $Y$. In particular, we revisit some known examples and we obtain some new Calabi-Yau varieties with small Hodge numbers. The groups $G$ are subgroups of the automorphism groups of $X$, which is described in terms of the automorphism group of the two del Pezzo surfaces.
\end{abstract}
\tableofcontents


\section{Introduction}
In \cite{TianYau} and \cite{Yau} Tian and Yau discover a new Calabi-Yau manifold with Euler characteristic equal to -6. Let us briefly explain their seminal example. To begin with, they consider the product $X$ of two cubic Fermat surfaces in $\PDD$. Next, they pick a smooth hyperplane section $Y$ in $X$, which is invariant with respect to a group $G$ isomorphic to the cyclic group of order three. By adjunction and by Lefschetz's Hyperplane Theorem, $Y$ turns out to be a smooth Calabi-Yau threefold, i.e., a smooth compact K\"{a}hler threefold with trivial canonical bundle and no holomorphic $p$-forms for $p=1,2$.  The Euler characteristic of $Y$ is $-18$ and the two significant Hodge numbers $h^{1,1}(Y)$ and  $h^{1,2}(Y)$ are $14$ and $23$, respectively. To reduce to Euler characteristic and the Hodge numbers, Tian and Yau take the quotient of $Y$ with respect to $G$ that turns out to act freely on it. The quotient manifold $Y/G$ is a Calabi-Yau variety with Hodge numbers $h^{1,1}=6$ and $h^{1,2}=9$.
\vspace{4mm}

\noindent In recent years, physicists have focused on Calabi-Yau manifolds with small Hodge numbers: see, for instance, \cite{Braun}, \cite{AltroCandelas}, \cite{Candelas}, \cite{Davies} and \cite{Freitag}. In fact, imagine to plot the distribution of Calabi-Yau varieties on a diagram with variables the Euler characteristic $\chi(Y)$ (on the horizontal axis) and the height $h(Y):=h^{11}(Y)+h^{12}(Y)$ (on the vertical axis). Fix a pair $(\chi_0, h_0)$ of positive integers such that $\chi_0$ is even and $-2h_0 \leq \chi_0 \leq 2h_0$. For $h_0 \leq 30$, it turns out that there are still a lot of missing examples of Calabi-Yau varieties with Euler characteristic $\chi_0$ and height $h_0$. The example in \cite{TianYau} is even more significant because the Euler characteristic is $-6$. In general, special attention is given to those Calabi-Yau manifolds that have Euler characteristic $6$ in absolute value since they correspond to three-generation families (see, for instance, \cite{Candelas}).
\vspace{4mm}

\noindent Remarkably, the example in \cite{TianYau} can be generalized in the following way. The two cubic Fermat surfaces are examples of degree three del Pezzo surfaces, i.e., smooth surfaces with ample anticanonical divisor which can be obtained as the blow-up of $\PD$ at six points in general position. A first generalization in this direction was given by Braun, Candelas and Davies in \cite{Candelas}. In that paper, they discover a new Calabi-Yau manifold with Euler characteristic $-6$ and small Hodge numbers. They replace the two Fermat surfaces in $\PDD$ by two del Pezzo surfaces of degree six and come up with a group of order twelve that acts freely on a suitable hyperplane section  of the product.
\vspace{4mm}

\noindent In this paper we generalize the examples mentioned above even further and we put them in a more general context. Indeed, let us consider two suitable smooth del Pezzo surfaces $S_1$ and $S_2$. The product $X$ is a smooth Fano fourfold, i.e., $-K_X$ is ample. In $X$ we pick a smooth threefold $Y$ which is in $|-K_X|$. As pointed out by the example in the Introduction in \cite{AltraVoisin} this requires some work: in fact, for some choice of the two del Pezzo surfaces it is not even possible. Moreover, we pick a finite group $G$ in $\Aut(S_1 \times S_2)$ that acts freely on $Y$ so that the quotient variety is a Calabi-Yau manifold. Since the Euler characteristic $\chi(Y)$ is negative, it is easy to verify that the height of $Y/G$ is less than the height of $Y$ for any non-trivial group $G$. Within this set-up, we obtain the two examples mentioned above; further, we find new Calabi-Yau manifolds with small Hodge numbers. The smoothness and the free action of $G$ on a suitable $Y$ are proved as follows. We pick a group $G$ that has only finitely many fixed points on $X$. We decompose the representation of $G$ on $H^0(X, -K_X)$ as a direct sum $\oplus V_{i}$ of irreducible subrepresentation. We consider a subspace $W$ such that for every $g\in G$ and every $s\in W$, $g^{*}(s)=\lambda_{g} s$ for some $\lambda_{g}\in \C^*$, i.e. for every $g\in G$, $W$ is an eigenspace for $g^*$.We pick a section $s\in W$, if there are some, so that the corresponding zero locus does not intersect the fixed locus of $G$. Next, we look at the base points of the subsystem $W\leq H^0(X, -K_X)$. In case there are some, we take a generic section and prove that the base points are smooth. This is done by direct computation with MAPLE. A Bertini-type argument yields the existence of a smooth threefold $Y$ in $X$ on which $G$ acts freely.
\vspace{4mm}

\noindent In Section \ref{newexamples} we present the examples we obtain case by case. Except for the last subsection of that Section, all the examples have height less than $20$. Unfortunately, we do not obtain any new three-generation manifolds, i.e., a manifold with $|\chi(Y)|=6$. Moreover, in Section \ref{list}, you may find all the examples of quotients of Calabi-Yau threefolds $Y$ embedded in $S_1 \times S_2$ by groups which are of maximal order. In other words, we take the quotient by a group $H \leq \Aut(S_1 \times S_2)$ such that the restriction to $Y$ yields a free action and $H$ can not have order greater than the groups used. Finally, we investigate the height of the quotient variety. In several cases, we are able to say that the height for the quotient threefold is the least possible within this framework.
\vspace{4mm}

\nocite{Voisin}

\nocite{Matsuki}

\nocite{DolIsk}

\noindent The following picture represents the tip of the distribution of the Calabi-Yau manifold with respect to the Hodge numbers. The diagonal axis are $h^{1,1}(Y)$ and $h^{1,2}(Y)$ whereas the horizontal and the vertical axis are $\chi(Y)$ and $h(Y)$, respectively. We plot only the known manifolds with height less or equal than $31$. The solid dots correspond to quotients found in this paper. The blue rings represent the ones known until now (with respect to the data collected in \cite{Braun}, \cite{AltroCandelas}, \cite{Candelas} and \cite{Davies}). The black rings are quotients by groups whose order is maximal.  From the picture below, we can summarize our results as follows. The dots $(3,5), (2,7)$ and $(5,13)$ represent NEW Calabi-Yau threefolds. There exists a Calabi-Yau manifold corresponding to the pair $(1,5)$ with non-abelian fundamental group: see \cite{AltroCandelas}.  Our example in Section \ref{p14} has abelian fundamental group isomorphic to the product of the cyclic group of order two and that of order eight. Moreover, we come up with a Calabi-Yau manifold with Hodge numbers $(2,11)$ (cf. \eqref{p2p2}), which are the same as those described in \cite{AltroCandelas}.  Finally, we  construct other varieties with greater height (see Section \ref{other}) but they correspond to existing dots in the picture below. In all the cases where other Calabi-Yau manifolds already exist, it would be interesting to know whether our examples are isomorphic to those or not.

\begin{center}
\includegraphics*[width=.99\textwidth]{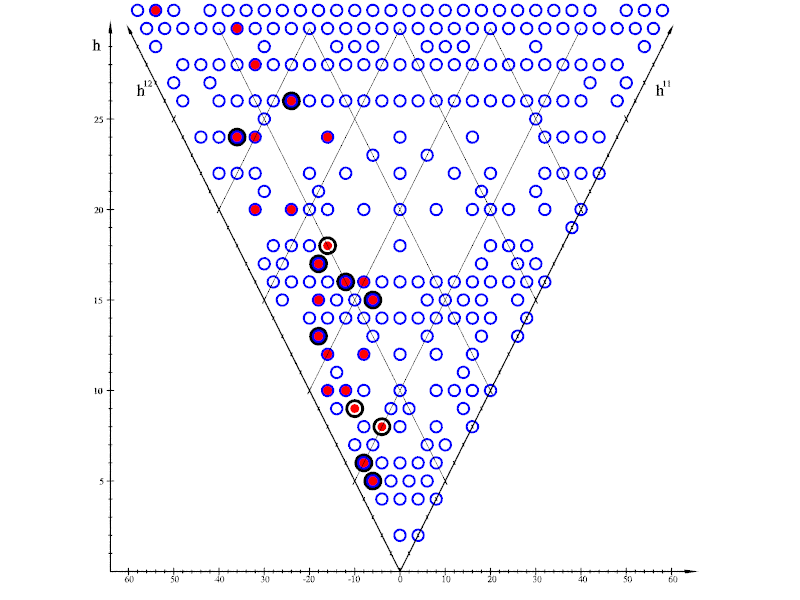}
\end{center}

\noindent In some cases, it is not possible to consider non-trivial quotients with our method. In fact, we prove, for instance, that there does not exist a Calabi-Yau variety which is the quotient by a group of order seven of a smooth anticanonical section $Y$ in a product of two del Pezzo surfaces of degree two.  This type of results is collected in Section \ref{results}. To prove them, we use the following theorem which is proved in Section \ref{relation}. For this purpose, we first use some  Mori theorem of Fano fourfolds which are products of two Fano varieties. Second, we also recall that for low degree del Pezzo surfaces are toric varieties. Thus, we apply a theorem due to Demazure (later generalized by D. Cox in \cite{Cox}) on the structure of the automorphism group of toric varieties. More specifically, the following holds (see Section \ref{relation}).

\begin{thma}
\label{THM:AUTdPs1}
Let $S_1$ and $S_2$ be two del Pezzo surfaces. Then
\begin{itemize}
\item If $S_1\neq S_2$, $\Aut(S_1\times S_2)=\Aut(S_1)\times\Aut(S_2)$;
\item If $S_1=S_2\neq \P^1\times \P^1$, $\Aut(S^{\times 2})=\Aut(S)^{\times 2}\ltimes \Z_2$;
\item If $S_1=S_2= \P^1\times \P^1$, $\Aut((\P^1)^{\times 4})=\Aut(\P^1)^{\times 4}\ltimes S_4,$
where $S_4$ is the symmetric group with $24$ elements.
\end{itemize}
\end{thma}

\noindent {\bf Acknowledgments.} During the preparation of this work, we asked some questions and suggestions to various people we kindly acknowledge: Alberto Alzati, Cinzia Casagrande, Philip Candelas, Igor Dolgachev, Bert van Geemen, Grzegorz Kapustka, Michal Kapustka, Antonio Lanteri and Gian Pietro Pirola.  This work was partially supported by MIUR and GNSAGA.

\newpage

\section{Preliminaries}

We say that a complex surface $S$ is a del Pezzo surface if it is projective, smooth, simply-connected and the anticanonical divisor $-K_S$ is ample. Examples of del Pezzo surfaces are blow-ups of the projective plane in a finite set $\Delta$ of $0\leq n<9$ points in general position and $\P^1\times \P^1$. As proved in \cite{Dolgachev}, this list is exhaustive. We often write $dP_d$ to mean a del Pezzo surface that is obtained by blowing up $9-d$ points of $\P^2$ that are in general position. Let $S=\Bl_{\Delta}\P^2$. We can identify $H^0(S,-K_S)$ with the vector space of the homogeneous polynomials of degree $3$ with variables  $\left\lbrace x_0,x_1, x_2\right\rbrace$ such that $f(P)=0$ for all $P\in\Delta$. It is easy to show that $h^0(S,-K_S)=d+1$ if $S=dP_d$. Moreover, if $k=9-d$ then $$-K_S=3\pi^*H-\sum_{i=1}^k E_i,$$ where $H$ is the hyperplane divisor on the projective plane and the $E_i$'s are the exceptional divisors.   Thus,  $K_S^2=9-k=d$. For $ d\geq 3$ we have that $-K_{S}$ is very ample. For $d=2$ the anticanonical system $|-K_S|$ gives a $2:1$ map of $S$ in $\P^2$  branched along a smooth quartic. For $d=1$ the anticanonical model of $S$ is a finite cover of degree two of a quadratic cone $Q$ ramified over a curve $B$ in the linear system $|{\mathcal O}_Q(3)|$.
\vspace{4mm}

\noindent Suppose that $Y$ is a Calabi-Yau threefold and that $G$ is a group that acts freely on $Y$. Then it is well known that the quotient $Y/G$ has a canonical complex structure such that the projection on the quotient is holomorphic. Furthermore, the quotient map is a local isomorphism.


\begin{thm}
\label{THM:Quo}
If the action of $G$ is free then $Y/G$ is also a Calabi-Yau threefold. Moreover, the quotient is projective.
\end{thm}

\begin{proof}
Take $g\in G\setminus \left\lbrace\id\right\rbrace$.  The manifold $Y$ is a Calabi-Yau threefold,  so $$h^{1,0}(Y)=h^{2,0}(Y)=0,\quad h^{3,0}(Y)=1.$$
There exists $\omega \in H^{3,0}(Y)$ such that $\omega_P\not\equiv 0$ for all $P\in Y$ (this is equivalent to $K_{Y}\equiv 0$).
We want to show that $g^*\omega=\omega$.
The maps
$$g^*:H^{p,0}(Y)\longrightarrow H^{p,0}(Y)$$
are zero for $p=1,2$ whereas for $p=0$, $g^*$ is the identity. We apply the Holomorphic Lefschetz Fixed Point formula, which in this case reads as follows:
$$0=1-0+0-\Tr(g^*:H^{3,0}(Y)\longrightarrow H^{3,0}(Y)).$$
Since $h^{3,0}(Y)=1$ ($Y$ is a Calabi-Yau manifold) we get $g^*=\id$ for $p=3$ and for all $g\in G$. Thus the action of $G$ on $H^{3,0}(Y)$ is trivial.
We have the following isomorphism (\cite{Bogomolov}, p. 198):
$$H^{p,q}(Y/G)\simeq H^{p,q}(Y)^G;$$
hence $H^{3,0}(Y/G)\simeq H^{3,0}(Y)^G=H^{3,0}(Y)$ and there exists a holomorphic $3-$form $\tilde{\omega}$ on $Y/G$ such that $\pi^*\tilde{\omega}=\omega$ and, as $\pi$ is a local isomorphism,  $\tilde{\omega}_P\neq0$ for all $P\in Y/G$. This is equivalent to $K_{Y/G}\equiv 0$.
Finally, using $h^{p,0}(Y/G)=h^{p,0}(Y)^G$ one has $h^{1,0}(Y/G)=h^{2,0}(Y/G)=0$ and this concludes the proof.
As for the projectivity of $Y/G$, see, for example, \cite{harris}, p. 127.
\end{proof}

\vspace{4mm}

\noindent We will adopt the following framework. We will take two del Pezzo surfaces $S_1$ and $S_2$, their product $X=S_1\times S_2$, which is a Fano fourfold, and a smooth element $Y$ of $|-K_X|$.

\noindent First of all, we will define a number $M(S_1,S_2)$ which bounds the maximum order of a finite group acting freely on $Y$ and which depends only on the degree of $S_1$ and $S_2$.
\begin{deff}
\label{emme}
Let $M(S_1,S_2)$ to be the positive greatest common divisor of $\chi(Y)/2$ and $\chi(-\iota^*K_X))$, where $\iota : Y \rightarrow X$ is the embedding of $Y$ in $X$.
\end{deff}
\noindent Notice that if $Y \subset S_1\times S_2$ is a Calabi-Yau threefold and $G$ is a finite group that acts freely on $Y$, then $|G|$ divides $M(S_1,S_2)$.
\vspace{4mm}

\noindent With the definition of $M(S_1, S_2)$ in mind, we will search for a group $G$ with the following properties:
\begin{enumerate}[(a)]
\item $G$ is a subgroup of $\Aut(S_1\times S_2)$;
\item $|G|=M(S_1,S_2)$.
\end{enumerate}
Note that if $\Fix(G) \subset X$ contains a curve $L$, by the Nakai-Moishezon criterion of ampleness, $-K_X\cdot L>0$, and since $Y=-K_X$ we will have some fixed points on $Y$. Hence it's necessary to choose groups whose action on $X$ has at most a finite number of fixed points.
\vspace{4mm}

\noindent Finally, Let $m(S_1,S_2,Y)$ be $$\max\left\lbrace |G|\quad|\quad g(Y)=Y\,\,\forall g\in G \mbox{ and satisfies } (a) \mbox{ and } (b)\right\rbrace.$$
We anticipate that there are cases in which $M(S_1,S_2)>1$ but the only group with these requests is the trivial group (that is $m(S_1,S_2,Y)=1$ for all $Y$).


\section{Necessary Conditions}
Assume that $S_1$ and $S_2$ are smooth projective surfaces and $Y$ is a Calabi-Yau threefold embedded in $X=S_1\times S_2$. Then the following result holds:

\begin{thm}
\label{THM:Chi}
The Euler characteristic of $Y$ is $$-2K_{S_1}^2K_{S_2}^2.$$
\end{thm}

\begin{proof}
By the exact sequence of vector bundles
$$0\rightarrow T_Y \rightarrow T_X \rightarrow N_{Y/X}\rightarrow 0$$
and, as $Y$ is a Calabi-Yau manifold ( which implies $c_1(Y)=0$), we have:
$$(1+\c_2(Y)+\c_3(Y))\cdot (1+\c_1(N_{Y/X})) = \iota^*(1+\c_1(X)+\c_2(X)+\c_3(X)+\c_4(X))$$
and, in particular,
$$c_1(N_{Y/X})=\iota^*c_1(X),\quad c_2(Y)=\iota^*c_2(X)\quad \mbox{ and}$$ $$c_3(Y)=\iota^*c_3(X)-c_2(Y)c_1(N_{Y/X}).$$
Using the fact that $X$ is a product of surfaces we have
$$c_1(X)=c_1(S_1)+c_1(S_2),\quad c_2(X)=c_2(S_1)+c_2(S_2)+c_1(S_1)c_1(S_2)$$
and
$$c_3(X)=c_2(S_1)c_1(S_2)+c_2(S_2)c_1(S_1).$$
Hence, by the identification $H^6(Y,\Z)\simeq \Z$, we have
$$c_3(Y)=\iota^*(c_3(X)-c_2(X)c_1(X))=c_3(X)c_1(X)-c_2(X)c_1(X)^2=$$
$$c_2(S_1)c_1(S_2)^2+c_2(S_2)c_1(S_1)^2-c_2(S_1)c_1^2(S_2)-c_2(S_2)c_1(S_1)^2-2c_1(S_1)^2c_1(S_2)^2=$$
$$=-2c_1(S_1)^2c_1(S_2)^2=-2K_{S_1}^2K_{S_2}^2.$$
\end{proof}

\noindent Now, assume $Y$ is asmooth ample divisor in $X$. Thus, the following isomorphisms hold:
$$
H^2(Y,\Z) \simeq H^2(X, \Z) \simeq H^2(S_1, \Z) \oplus H^2(S_2,\Z).
$$

\noindent For any divisor class $D \in H^2(Y/G,\Z)$ denote by $D_1$ and $D_2$ divisors classes such that $\pi^*(D)=D_1+D_2$, where $\pi$ is the projection of $Y$ onto the quotient. Finally,  we denote by $K_i$ the divisor classes  such that $K_X=K_1+K_2$. Then the following holds.

\begin{thm}
\label{THM:RR}
Let $G$ be a group that acts freely on $Y$. Then for  any $D \in H^2(Y/G,\Z)$ and $D_1, D_2$ as above,  we have
$$\chi(D)=-\frac{D_1D_2(D_1K_2+D_2K_1)}{2|G|}-
\frac{\chi(\O_{S_1})K_2D_2+\chi(\O_{S_2})K_1D_1}{|G|}.$$
\end{thm}

\begin{proof}
We recall that the Riemann-Roch formula for the Calabi-Yau threefold $Y/G$ is
$$\chi(D)=\frac{D^3}{6}+\frac{c_2(Y/G)D}{12}.$$
The action of $G$ is free, hence
$$|G|D^3=\pi^*(D)^3 \quad \mbox{ and } \quad |G|c_2(Y/G)D=c_2(Y)\pi^*(D).$$ This yields
$$\pi^*(D)^3=(D_1+D_2)^3(c_1(S_1)+c_1(S_2))=$$
$$=3D_1^2D_2c_1(S_2)+3D_1D_2^2c_1(S_1)=-3D_1D_2(D_1K_2+D_2K_1).$$
In a similar way, we obtain
$$c_2(Y)\pi^*(D)=-(\chi(S_1)+K_1^2)K_2D_2-(\chi(S_2)+K_2^2)K_1D_1.$$
Merging these results and using N\"{o}ther formula\footnote{$\chi({\mathcal O}_S)=\frac{K_S^2+\chi(S)}{12}.$}, we complete the proof.
\end{proof}

\noindent We focus our attention on a particular divisor on the quotient: a divisor $D$ such that $\pi^*D=-\iota^*K_X=$. Such a divisor always exists because the canonical divisor is $G$-invariiant for any gruoup of automorphisms $G$.
We can specialize the previous formula for $nD$ obtaining
$$\chi(nD)=n^3\frac{K_1^2K_2^2}{|G|}+n\frac{\chi(\O_{S_1})K_2^2+ \chi(\O_{S_2})K_1^2}{|G|}=\frac{\chi(-n\iota^*K_X)}{|G|}.$$
Hence $|G|$ has to divide $\chi(-n\iota^*K_X)$ for all\footnote{One could easily check that
$$|G| \mbox{ divides } \chi(-n\iota^*K_X)\quad \forall n\in \Z \Longleftrightarrow |G| \mbox{ divides }  \chi(-\iota^*K_X).$$} $n$. We can obtain a similar condition using Theorem \ref{THM:Chi}: the Euler characteristic of the quotient $Y/G$ of $Y$ by a finite group $G$ that acts freely is the Euler characteristic of $Y$ divided by the order of the group. Moreover, it is known that a Calabi-Yau threefold has even Euler number so we obtain that $|G|$ must divide $\chi(Y)/2$. This gives a motivation to Definition \ref{emme}.

\vspace{4mm}

\noindent The following table gives the values of $M(S_1,S_2)$ for every distinct values of degrees of $S_1$ and $S_2$, with $S_1$ and $S_2$ del Pezzo surfaces - distinguishing the case $dP_8$ and $\P^1\times \P^1$.
\vspace{4mm}

\begin{displaymath}
\begin{array}[c]{c||cccccccccc}
M(S_1,S_2) & \P^2 & \P^1\times\P^1 & dP_8 & dP_7 & dP_6 & dP_5 & dP_4 & dP_3 & dP_2 & dP_1 \\ \hline\hline
\P^2 &9&1&1&1&3&1&1&3&1&1\\
\P^1\times \P^1 &1&16&16&1&2&1&4&1&2&1\\
dP_8 &1&16&16&1&2&1&4&1&2&1\\
dP_7 &1&1&1&7&1&1&1&1&1&1\\
dP_6 &3&2&2&1&12&1&2&9&4&1\\
dP_5 &1&1&1&1&1&5&1&1&1&1\\
dP_4 &1&4&4&1&2&1&8&1&2&1\\
dP_3 &3&1&1&1&9&1&1&3&1&1\\
dP_2 &1&2&2&1&4&1&2&1&4&1\\
dP_1 &1&1&1&1&1&1&1&1&1&1
\end{array}
\end{displaymath}

\noindent For example,  if $X=dP_2\times dP_5$ ($M(dP_2,dP_5)=1$) it isn't possible to find a pair $(Y,G)$ with $Y$ embedded in $X$ and $\id\neq G\leq \Aut(Y)$ that acts freely on $Y$. If we choose $X=dP_5\times dP_5$ ($M(dP_5,dP_5)=5$) a pair $(Y,\Z_5)$ with $\Z_5$ without fixed points \textit{might} exist.
\vspace{4mm}

\noindent The self-intersection of $-K_S$, where $S$ is a del Pezzo surface,  is positive and is equal to its degree and this, using Theorem \ref{THM:Chi}, means that $\chi(Y)<0$ regardless of the choice of which surfaces we are using. Therefore, by recalling that the action of $G$ is free, we have that the {\em height} $h:= h^{1,1} + h^{1,2}$ of $Y$ and that of $Y/G$ satisfy the following inequality:
\begin{eqnarray*}
h(Y/G)=h^{1,1}(Y/G)+h^{1,2}(Y/G)&=& 2h^{1,1}(Y/G)-\frac{\chi(Y)}{2|G|}\\
&=& 2h^{1,1}(Y)^G+\frac{|\chi(Y)|}{2|G|} \\
&<&2h^{1,1}(Y) + \frac{|\chi(Y)|}{2}= h(Y).
\end{eqnarray*}

\vspace{4mm}

\noindent By finding a group whose order is maximal   - and such that the dimension $h^{1,1}(Y)^G$ of the invariant part of $H^{1,1}(Y)$ is the smallest possible - we obtain the least possible height for the quotient.
\vspace{4mm}

\noindent In the following sections we give some examples (both known and new) and some results of non-existence.


\section{Known Examples}
With the following examples we revisit some known examples in the framework presented. The first one is due to Braun, Candelas and Davies and can be found in \cite{Candelas}. The second one is due to Tian and Yau and is presented in \cite{Yau} and \cite{TianYau}.


\subsection{$dP_6\times dP_6$ with maximal order $12$}

There is a unique del Pezzo surface of  degree $6$ and this surface can be obtained as the complete intersection of two global sections of $\O_{\P^2\times \P^2}(1,1)$. Explicitly we can take $S$ to be the surface in $\P^2\times \P^2$ given by the equations $$f=x_{10}x_{20}-x_{11}x_{21}\quad\mbox{ and }\quad g=x_{10}x_{20}-x_{12}x_{22},$$
where $x_{ij}$ is the $j$-th coordinate on the $i$-th copy of $\P^2$. In this way, $S$ is the surface obtained by blowing up the points $P_0=(1:0:0),P_1=(0:1:0)$ and $P_2=(0:0:1)$ of $\P^2$ and the exceptional divisors $E_i$ are given by
$$E_0:=V(x_{11},x_{12},x_{20}),E_1:=V(x_{10},x_{12},x_{21}) \mbox{ and } E_2:=V(x_{10},x_{11},x_{22}).$$
We define $S_1=S_2=S$ and embed $X=S\times S$ in $(\P^2)^4$ using $x_{i0},x_{i1}$ and $x_{i2}$ as projective coordinates of the $i-$th $\P^2$ for $i=1,2,3,4$. Let $P$ be the point
$$
P:=((x_{10}, x_{11}, x_{12}), (x_{20}, x_{21}, x_{22}), (x_{30}, x_{31}, x_{32}), (x_{40}, x_{41}, x_{42})).
$$
Consider the automorphism of $X$ defined by
$$g_3(P)=((x_{12}:x_{10}:x_{11}),(x_{22}:x_{20}:x_{21}),(x_{31}:x_{32}:x_{30}),(x_{41}:x_{42}:x_{40}))$$
$$g_4(x_1,x_2,x_3,x_4)=(x_4,x_3,x_1,x_2).$$
It is easy to check that $g_3^3=g_4^4=\id$ and $g_4g_3=g_3g_4^2$ hence $$G=<g_3,g_4>\simeq \Z_4\ltimes\Z_3:=:\Dic_3,$$ which is called the dicyclic group of order $12$.
The set of the fixed points $\Fix$ of $G$ is given by the union of
$$\Fix(g_3)=\left\lbrace (Q_1 , Q_2) \quad |\quad Q_1, Q_2\in\left\lbrace (1:a:a^2)\times(1:a^2:a) \quad |\quad a^3=1\right\rbrace \right\rbrace $$
and
$$\Fix(g_4^2)=\left\lbrace (T\times T\times Q\times Q\quad |\quad T,Q\in\left\lbrace (1:\pm1:\pm1)\right\rbrace \right\rbrace$$
so we have a total of $25$ fixed points.
\vspace{4mm}

\noindent We are looking for a global section $s$ of $\O_{X}(-K_X)$ that is $G-$invariant and whose zero-locus $V(s)$ is smooth and doesn't intersect $\Fix$. We have an exact sequence
$$0\rightarrow <f,g> \hookrightarrow H^0(\P^2\times \P^2,\O(1,1))\twoheadrightarrow H^0(S,-K_S)\rightarrow 0$$
with the surjection given by the inclusion $\iota:S\rightarrow \P^2\times\P^2$.
Hence, we have a surjection $$H^0((\P^2)^4,\O(1,1,1,1))\twoheadrightarrow H^0(X,-K_X)$$ with kernel given by $$<f_1,g_1>\cdot H^0((\P^2)^4,\O(0,0,1,1))+<f_2,g_2>\cdot H^0((\P^2)^4,\O(1,1,0,0)).$$
\ifVERSIONEPROLISSA
\else 
The representation of $\Dic_3$ in $H^0(X,-K_X)\simeq\C^{49}$ has an invariant space $H^0(X,-K_X)^G$ of dimension $5$.
By direct inspection, we have checked that the generic invariant section $s$ doesn't intersect $\Fix$ and is smooth. Then $Y=V(s)$ is a Calabi-Yau threefold with a free action of $\Dic_3$.\vspace{4mm}
\fi 

\noindent If we call $R$ the representation of $\Dic_3$ in $H^2(Y,\C)\simeq H^2(X,\C)$ given by $g_i\mapsto g_i^*\in \GL(H^2(X,\C))\simeq\GL(H^2(S,\C)\oplus H^2(S,\C))\simeq \GL(\C^8)$ we have
$$R(g_3)\leftrightsquigarrow A_3:=\left[ \begin {array}{cccc|cccc}
0&1&0&0&0&0&0&0 \\
0&0&1&0&0&0&0&0 \\
1&0&0&0&0&0&0&0 \\
0&0&0&1&0&0&0&0 \\\hline
0&0&0&0&0&0&1&0 \\
0&0&0&0&1&0&0&0 \\
0&0&0&0&0&1&0&0 \\
0&0&0&0&0&0&0&1
\end {array} \right]$$
and
$$R(g_4) \leftrightsquigarrow A_4:=\left[ \begin {array}{cccc|cccc}
0&0&0&0&1&0&0&0 \\
0&0&0&0&0&1&0&0 \\
0&0&0&0&0&0&1&0 \\
0&0&0&0&0&0&0&1 \\\hline
 0&-1&-1&-1&0&0&0&0 \\
-1& 0&-1&-1&0&0&0&0 \\
-1&-1& 0&-1&0&0&0&0 \\
 1& 1& 1& 2&0&0&0&0
\end {array} \right],$$
where we used the base $$\left\lbrace H_1,E_{10},E_{11},E_{12},H_2,E_{20},E_{21},E_{22} \right\rbrace $$ where
$\pi^*_i(E_{j})=E_{ij}$ and $\pi^*_i(\pi^*H)=H_{i}$.
Hence $\Dim H^2(Y,\C)^G=1$,  so we have $h^{1,1}(Y/G)=1$. By Theorem \ref{THM:Chi} we know that $\chi(Y)=-72$ and then $\chi(Y/G)=-6$ because the action is free. In conclusion, you find below  the Hodge diamond of $Y/G$
$$\xymatrix@R=8pt@C=6pt{
   &   &   & 1   \\
   &   & 0 &   & 0 \\
   & 0 &   & 1 &   & 0 \\
1  &   & 4 &   & 4 &   & 1, \\
   & 0 &   & 1 &   & 0 \\
   &   & 0 &   & 0 \\
   &   &   & 1
}$$
where $h(Y/G)=5$. Note that, because $h^{1,1}(Y/G)=1$, this example achieves the minimum of the height for the quotient $Y/G$, where $G$ is isomorphic to the $Dic_3$ and $Y$ is as above.  It is interesting to note that taking
$$g_3'(P)=((x_{12}:x_{10}:x_{11}),(x_{22}:x_{20}:x_{21}),(x_{32}:x_{30}:x_{31}),(x_{42}:x_{40}:x_{41}))$$
the group $G'$ spanned by $g_4$ and $g_3'$ is cyclic of order $12$ and a generator is $g_3'g_4:=g_{12}$. Following the same argument as the previous case it can be shown that exist a Calabi-Yau $Y$ such that $G'$ acts on $Y$ freely. The quotient $Y/G'$ is hence again a Calabi-Yau and has the same Hodge diamond as $Y/G$. However these two manifolds aren't even diffeomorphic because $\Pi_1(Y/G)\simeq G\simeq \Dic_3\not\simeq \Z_{12}\simeq G'\simeq \Pi_1(Y/G')$.


\subsection{$dP_3\times dP_3$ with maximal order $3$}

Suppose $S_1$ and $S_2$ del Pezzo surfaces of degree $3$. Then $-K_{S_i}$ is very ample and gives an embedding in $\P^3$. The surface obtained is a cubic (is called anticanonical model of $S_i$) and all smooth cubic surfaces in $\P^3$ can be obtained in this way.
\vspace{4mm}

\noindent Set $f_1:=x_0^3+x_1^3+x_2^3+x_3^3$ and $f_2:=y_0^3+y_1^3+y_2^3+y_3^3$ and consider the Fermat surfaces $S_i:=V(f_i)$. Denote, as usual, $X=S_1\times S_2\subset \P^3\times \P^3$ and consider
the automorphism given by
$$\varphi(x,y)=((x_1 : x_2 : x_0 : \omega x_3),(y_1 : y_2 : y_0 : \omega^2 y_3))$$
where $\omega\neq 1$ is a fixed root of $z^3-1$. The group $G=<\varphi>$ is cyclic of order $3$; hence we have
$$\Fix(\varphi)=\Fix(G)=\left\lbrace ((1:\omega^2 : \omega : c),(1:\omega : \omega^2 : d))\,\,|\, d^3=c^3=-3\right\rbrace.$$
\vspace{4mm}

\noindent There is an isomorphism
$$H^0(\P^3\times\P^3,\O(1,1))\simeq H^0(X,-K_X)$$
so we have to study the polynomial of bidegree $(1,1)$. The action of $\Z_3$ on $X$ gives a representation of $\Z_3$ in $H^0(X,-K_X)$ and a basis for the invariant space is
$\left\lbrace G_0,G_1,G_2,G_3,G_4,G_5 \right\rbrace $ where
$$G_0 = \omega x_3y_0+\omega^2 x_3y_1+x_3y_2,\quad  G_1 = \omega^2 x_0y_3+\omega x_1y_3+x_2y_3,$$
$$G_2 = x_0y_1+x_1y_2+x_2y_0,\quad G_3 = x_0y_2+x_1y_0+x_2y_1,$$
$$G_4 = x_0y_0+x_1y_1+x_2y_2 \quad\mbox{ and }\quad G_5 = x_3y_3.$$
By direct computation, one can check that the generic section $s$ doesn't intersect $\Fix(\Z_3)$ hence the action of $G$ restricted to $V(s)$ is free. For example, taking $s$ to be $G_4+G_5=x_0y_0+x_1y_1+x_2y_2+x_3y_3$ gives a section whose zero locus $Y$ is smooth and $Y\cap \Fix(\Z_3)$ is empty.
\vspace{4mm}

\noindent Assume $\varphi\in \Aut(S_1)\times \Aut(S_2)$ with $o(\varphi)=3$. By the Lefschetz fixed-point theorem, one can show that
$$h^{1,1}(Y)^G = 2+\frac{2}{3}\left(\chi(\Fix(\pi_1\circ \varphi))+\chi(\Fix(\pi_2\circ \varphi))\right),$$
where $\pi_i:X \rightarrow S_i$ is the projection onto the $i-$th factor of the product $X$. In fact, by Lefschet's Hyperplane Theorem, the group $H^{1,1}(Y)$ is isomorphic to $H^2(X)$. The dimension of the space of invariants with respect to $G$ is equal to the traces of the homomorphisms induced on the second cohomology group of $X= S_1 \times S_2$ by the elements of $G$. By linear algebra and the K\"{u}nneth formula, the traces on the cohomology groups of the product $X$ is the sum of the traces on the cohomology on the factors $H^2(S_i)$ for $i=1,2$. These traces can be computed via the Lefschetz fixed-point Theorem.  In this case we obtain $h^{1,1}(Y/G)=6$.
The same number could be obtained by studying the invariant space of $H^2(X,\C)$ with respect to the representation of $\Z_3$ given by
$$\varphi \mapsto \varphi^*\leftrightsquigarrow \left[ \begin{array}{ll}
A_1 & 0 \\
0 & A_2
\end{array}\right]$$
where $A_1$ and $A_2$ are respectively
$$\left[ \begin {array}{ccccccc} 0&0&0&1&0&0&0\\\noalign{\medskip}-1&-1
&-1&0&-1&-1&-2\\\noalign{\medskip}0&0&0&0&-1&-1&-1\\\noalign{\medskip}0
&-1&0&0&0&-1&-1\\\noalign{\medskip}0&0&-1&0&0&-1&-1
\\\noalign{\medskip}-1&0&0&0&0&-1&-1\\\noalign{\medskip}1&1&1&0&1&2&3
\end {array} \right]\,\mbox{ and }\,
\left[ \begin {array}{ccccccc} 0&0&0&-1&-1&0&-1\\\noalign{\medskip}0&0
&0&0&0&1&0\\\noalign{\medskip}-1&-1&-1&-1&-1&0&-2\\\noalign{\medskip}-
1&0&0&0&-1&0&-1\\\noalign{\medskip}0&-1&0&0&-1&0&-1
\\\noalign{\medskip}0&0&-1&0&-1&0&-1\\\noalign{\medskip}1&1&1&1&2&0&3
\end {array} \right].$$
By Theorem \ref{THM:Chi} we have $\chi(Y/\Z_3)=-18/3=-6$; so the Hodge diamond of the quotient is the foolowing one
$$\xymatrix@R=8pt@C=6pt{
   &   &   & 1   \\
   &   & 0 &   & 0 \\
   & 0 &   & 6 &   & 0 \\
1  &   & 9 &   & 9 &   & 1 \\
   & 0 &   & 6 &   & 0 \\
   &   & 0 &   & 0 \\
   &   &   & 1
}$$
In particular the height is $h(Y/\Z_3)=15$.
\vspace{4mm}

\noindent As shown in \cite{Dolgachev}, up to isomorphism of $\P^3$,  there are $3$ possible pairs $(f,G)$ where $f$ is a homogeneous polinomial of degree $3$ and $G$ is a group fixing $f$ of order $3$. One can show that $\Fix(f)$ is either one of the following: $3$ points or $6$ points, or one line. Thus, the least value that can be assumed by $\chi(\Fix{f})$ is $3$ if we exclude the case with one line of fixed points. Hence, the example presented here achieves the minumum for $h(Y/G)$.


\section{New Examples}\label{newexamples}

We present some new examples.


\subsection{$(\P^1\times\P^1)\times(\P^1\times\P^1)$ with maximal order $16$}\label{p14}

Take $S_1=S_2=\P^1\times \P^1$ and define $X$ to be $S_1\times S_2$. We begin to search for a group $H\leq (\Aut(\P^1))^4\ltimes S_4\leq \Aut(X)$ such that $|H|=8$ and $|\Fix(H)|<\infty$. Moreover, we want a section $s$ that is an eigenvector for the action of $H$ on $H^0(X,-K_X)$ and does not intersect $\Fix(H)$. After that, we try to extend $H$ to a group of order $16$ with the same properties.
\vspace{4mm}

\noindent Let $g\in (\Aut(\P^1))^4\ltimes S_4$ be an element of finite order. Without loss of generality, we can take $g$ of the form
$$\tilde{g}\circ \sigma:=\left(\ma{a_1},\ma{a_2},\ma{a_3},\ma{a_4}\right)\circ \sigma $$
where $\sigma\in {\mathfrak S}_4$ and $a_i \in \C^*$ for $i=1,2,3,4$.
\vspace{4mm}

\noindent If the order $o(g)$ of $g$ is $2$, we can choose $\sigma\in \left\lbrace \id,(12),(12)(34)\right\rbrace $. An easy check shows that
$$((x:y), (x:a_2y), (1:0), (1:0))$$ is a line of fixed points if $\sigma=(12)$ or $\sigma=(12)(34)$; so we must take $\sigma=\id$. The only possible case is $a_j=-1$ for which
$$g = g_2:=\left( \ma{-1},\ma{-1},\ma{-1},\ma{-1}\right) $$
and
$$\Fix(g_2)=\left\lbrace (P_1, P_2, P_3, P_4) \quad|\quad P_i\in\left\lbrace (1:0),(0:1) \right\rbrace \right\rbrace. $$
\vspace{2mm}

\noindent If $o(g)=4$, we can take $\sigma\in \left\lbrace \id,(12),(12)(34),(1234)\right\rbrace $. The automorphism $\sigma$ cannot be a permutation of order $4$. In fact, in this case $g^2$ would have a fixed line, as previously showed. Then, we have $\Fix(g)\subset \Fix(g^2)=\Fix(g_2)$.
Suppose $\sigma=\id$ or $\sigma=(12)$ and consider an eigenvector $s\in H^0(X,-K_X)=\O_X(2,2,2,2)(X)$. The condition $o(g)=4$ is then equivalent to $a_j^4=1$ for $\sigma=\id$ and $a_1^2a_2^2=a_3^4=a_4^4=1$ for $\sigma=(12)$.
Necessarily $g$ satisfies $g^2=g_2$ and this implies respectively $a_j^2=-1$ and $a_1a_2=a_3^2=a_4^2=-1$. One can see that for all $P\in \Fix(g_2)$ there exists a unique element $e_i$ of the usual basis of $\O_X(2,2,2,2)(X)$ such that $e_i(P)\neq 0$. For example, we have $$x_1^2x_2^2x_3^2x_4^2|_{(1:0)^4}=1\quad \mbox{ and} \quad x_1^2x_2^2x_3^2y_4^2|_{((1:0)^3, (0:1))}=1.$$
Then $s$ has to be an element of the eigenspace of both $x_1^2x_2^2x_3^2x_4^2$ and $x_1^2x_2^2x_3^2y_4^2$, but these have different eigenvalues ( $1$ and $a_4^2=-1$ respectively) so $s=0$. Suppose that $\sigma=(12)(34)$. The conditions $o(g)=4$ and $g^2=g_2$ show that $g$ has to be of the form $$\left( \ma{a_1},\ma{-a_1^{-1}},\ma{-a_3},\ma{-a_3^{-1}}\right)\circ(12)(34) $$
for some $a_3,a_4\in \C^*$.
\vspace{4mm}

\noindent Finally, take $g$ to be an automorphism of order $8$. Then $\sigma$ has to be a permutation of order\footnote{By this result, one can show that an element of order $16$ cannot exist in $(\Aut(\P^1))^4\ltimes S_4$ with the request we made. In fact, if such $g$ existed,  $g^2$ would have order $8$ and $g^2=(A_1,A_2,A_3,A_4)\circ \sigma^2$ with $\sigma^2$ permutation of order $4$. This is not possible for an element of ${\mathfrak S}_4$.} $4$. For example,  pick $\sigma=(1324)$ (that gives the following conditions on the $a_i$'s: $a_1a_2a_3a_4=-1$) and let $a_1=a_2=a_3=-a_4=1$.
A basis for $H^0(X,-K_X)$ is given by $\left\lbrace e_1,\dots e_{11} \right\rbrace $, where
\begin{align*}
e_1 = & x_1^2x_2y_2y_3^2x_4y_4+x_1y_1x_2^2x_3y_3y_4^2-x_1y_1y_2^2x_3^2x_4y_4+y_1^2x_2y_2x_3y_3x_4^2, \\
e_2 = & x_1^2x_2y_2x_3y_3y_4^2-x_1y_1x_2^2y_3^2x_4y_4+x_1y_1y_2^2x_3y_3x_4^2+y_1^2x_2y_2x_3^2x_4y_4, \\
e_3 = & x_1^2y_2^2x_3^2y_4^2+x_1^2y_2^2y_3^2x_4^2+y_1^2x_2^2x_3^2y_4^2+y_1^2x_2^2y_3^2x_4^2, \\
e_4 = & -x_1^2y_2^2x_3y_3x_4y_4+x_1y_1x_2y_2x_3^2y_4^2-x_1y_1x_2y_2y_3^2x_4^2+y_1^2x_2^2x_3y_3x_4y_4, \\
e_5 = & x_1^2x_2^2x_3^2y_4^2+x_1^2x_2^2y_3^2x_4^2+x_1^2y_2^2x_3^2x_4^2+y_1^2x_2^2x_3^2x_4^2, \\
e_6 = & x_1^2x_2y_2x_3^2x_4y_4+x_1^2x_2y_2x_3y_3x_4^2-x_1y_1x_2^2x_3^2x_4y_4+x_1y_1x_2^2x_3y_3x_4^2, \\
e_7 = & x_1^2x_2^2x_3^2x_4^2, \\
e_8 = & y_1^2y_2^2y_3^2y_4^2, \\
e_9 = & x_1^2y_2^2y_3^2y_4^2+y_1^2x_2^2y_3^2y_4^2+y_1^2y_2^2x_3^2y_4^2+y_1^2y_2^2y_3^2x_4^2, \\
e_{10} = & x_1^2x_2^2y_3^2y_4^2+y_1^2y_2^2x_3^2x_4^2, \\
e_{11} = & x_1y_1y_2^2x_3y_3y_4^2-x_1y_1y_2^2y_3^2x_4y_4+y_1^2x_2y_2x_3y_3y_4^2+y_1^2x_2y_2y_3^2x_4y_4.
\end{align*}

\noindent Now, we try to extend the group $H$.  Define $h$ to be the involution of $(\P^1)^4$ such that
$$(x_i:y_i)\longmapsto(y_i:x_i).$$
An easy check shows that $gh=hg$ and that the following hold:
$$\Fix(h)=\left\lbrace ((1:\pm1), (1:\pm1), (1:\pm1), (1:\pm1)) \right\rbrace$$
and
$$\Fix(g^4h)=\left\lbrace ((1:\pm i), (1:\pm i), (1:\pm i), (1:\pm i)) \right\rbrace.$$

\noindent For every $k\neq0,4$ we have $(g^kh)^2=g^{2k}h^2=g^{2k}$ so $\Fix(g^kh)\subset\Fix(g^4)=\Fix(g_2).$
This means that, defining $G$ to be the group generated by $g$ and $h$, $\Fix(G)$ is a finite set composed of $48$ points and $G\simeq \Z_8\times \Z_2$.
\vspace{4mm}

\noindent If we take $$s=\sum_{i=1}^{11}C_ie_i$$ and impose both $s(P)=1$ for all $P\in\Fix(g)$ and $h^*(s)=s$, we have the following conditions on the $C_i$'s:
$$C_3=C_5=C_7=C_8=C_9=C_{10}=1, C_1 = C_2, C_{11} = C_6, C_4 = 0.$$
By evaluating at the other fixed points, we obtain $4$ different non identically-zero linear-combinations of the $C_i$'s; so the generic invariant section does not intersect $\Fix(G)$.
For example, the section obtained by taking $C_1=1$ and $C_6=2$ fulfills all our requests. Moreover, it is smooth, so there exists a group of order $16=M(\P^1\times \P^1,\P^1\times \P^1)$ that acts freely on a Calabi-Yau threefold embedded in $(\P^1)^4$.
\vspace{4mm}

\noindent The representation of $G$ on $H^2(Y,\C)$ is given by
$$g\mapsto g^*\leftrightsquigarrow \left[ \begin {array}{cccc} 0&0&0&1\\\noalign{\medskip}0&0&1&0
\\\noalign{\medskip}1&0&0&0\\\noalign{\medskip}0&1&0&0\end {array}
 \right]\quad\mbox{ and }\quad h\mapsto h^* \leftrightsquigarrow \left[ \begin {array}{cccc} 1&0&0&0\\\noalign{\medskip}0&1&0&0
\\\noalign{\medskip}0&0&1&0\\\noalign{\medskip}0&0&0&1\end {array}
 \right] $$
so both $h$ and $g^4$ are trivial on $H^2(Y,\C)=H^2(\P^1,\C)^{\oplus 4}$.
\vspace{4mm}

\noindent This action has then a unique fixed class in $H^2(Y,\C)$ (the sum of the four $\P^1$'s). By Theorem \ref{THM:Chi}, we have $\chi(Y/G)=-128/16=-8$, so the Hodge diamond of the quotient $Y/G$ is the following one:
$$\xymatrix@R=8pt@C=6pt{
   &   &   & 1   \\
   &   & 0 &   & 0 \\
   & 0 &   & 1 &   & 0 \\
1  &   & 5 &   & 5 &   & 1 \\
   & 0 &   & 1 &   & 0 \\
   &   & 0 &   & 0 \\
   &   &   & 1
}$$
In particular, the height is $6$ and it's the least possible for a quotient of a Calabi-Yau in $(\P^1)^4$ because $h^{1,1}(Y/G)=1$.


\subsection{$dP_4\times dP_4$ with maximal order $8$}
\label{SEC:dP4xdP4}

As proved, for instance, in \cite{Dolgachev}, every del Pezzo surface of degree $4$ can be obtained as a complete intersection of two quadrics of $\P^4$. Moreover, one can choose the equations to be of the form
$$f=x_0^2+x_1^2+x_2^2+x_3^2+x_4^2\quad\mbox{ and }g=a_0x_0^2+a_1x_1^2+a_2x_2^2+a_3x_3^2+a_4x_4^2$$
where $a_i\neq a_j \in \C$ for $i\neq j$. We choose $$g=x_0^2-ix_1^2-x_2^2+ix_3^2$$
and $S_1\simeq S_2\simeq S=V(f,g)\subset \P^4$.
Let $r$ be the automorphism which sends $\left(x,y\right)$ to the point
$$\left((x_0: x_1: -x_2: x_3: -x_4), (y_0: y_1: -y_2: y_3: -y_4)\right).$$
Denote by $t$ the automorphism which sends $\left(x,y\right)$ to
$$\left((y_0: y_1: -y_2: -y_3: y_4), (x_0: x_1: x_2: x_3: x_4)\right).$$

Consider the groups $H=<r,t^2>\simeq\Z_2\times \Z_2$ and $G=<r,t>\simeq \Z_4\times \Z_2$.
\vspace{4mm}

\noindent By adjunction $-K_{S_1\times S_2}:=-K_X\simeq \O_X(5,5)\otimes \O_X(-4,-4)=\O_X(1,1)$. The morphism $\iota:S\times S \longrightarrow \P^4\times \P^4$ induces an isomorphism
$$\iota^*: H^0(\P^4\times \P^4,\O(1,1))\longrightarrow  H^0(S\times S, \O_X(1,1));$$
so we can use  $$\left\lbrace x_iy_j\right\rbrace_{0\leq i,j\leq 4}$$ as a basis of the space of sections of the anticanonical  bundle.
It is easy to see that the vector space $V$ spanned by $$\left\lbrace x_0y_0, x_1y_0, x_0y_1, x_1y_1, x_2y_2, x_3y_3, x_4y_4\right\rbrace$$ is such that for all $h\in H$ and for all $s\in V$, $h^*(s)=\lambda s$ for some $\lambda\in \C^*$.
By taking the generic section $s\in V$ and imposing $r^* s=t^* s = s$ (so that for every automorphism $g$ of $G$, $V$ is an eigenspace with respect to $g^*$), we obtain
$$s=A_1x_0y_0+A_3y_0x_1+A_3x_0y_1+A_4x_1y_1+A_7x_4y_4,$$
where $A_i \in \C$.
Let $a$ and $b$ be fixed roots of $2z^2+1+i$ and $2z^2+1-i$,  respectively. Then
$$Fix(r)=\left\lbrace (P, Q)\,|\, P,Q\in\left\lbrace (1:\pm a:0:\pm b: 0)\right\rbrace \right\rbrace$$
$$Fix(t^2)=\left\lbrace (P, Q)\,|\, P,Q\in\left\lbrace (\pm a:\pm b: 0: 0: 1)\right\rbrace \right\rbrace$$
and
$$Fix(rt^2)=\left\lbrace (P, Q)\,|\, P,Q\in\left\lbrace (\pm b:1:\pm a: 0: 0)\right\rbrace \right\rbrace.$$
To look for the fixed points of $G$ it suffices to know the fixed points of $r,t^2$ and $rt^2$. In fact, the following holds:
$$\Fix(t^3)=\Fix(t)\subseteq \Fix(t^2)=\Fix((rt)^2)\supseteq \Fix(rt)=\Fix(rt^3).$$
An easy check shows that for generic values of $A_1,A_3,A_4$ and $A_7$, the section $s$ does not intersect $\Fix(G)$.
\vspace{4mm}

\noindent We can check directly that the section corresponding to $A_1=1, A_3=-2, A_4=3$ and $A_7=1$ is smooth and doesn't intersect $\Fix(G)$; so there exists a Calabi-Yau threefold $Y$ embedded in $S_1\times S_2$ with $\Z_4\times \Z_2$ acting freely on $Y$.
\vspace{4mm}

\noindent We don't have an explicit description of a basis for $\Pic(Y)=\Pic(S_1)\oplus\Pic(S_2)\simeq \Z^{12}$, but we can use the Lefschetz Fixed Point fomula to get the traces we need to compute $h^{1,1}(Y)^G$.
For example, notice that $r=r_1\times r_2$ with $r_i\in\Aut(S_i)$; so the trace of $r^*: H^2(S_1\times S_2,\C)\rightarrow H^2(S_1\times S_2,\C)$ is equal to the sum of the traces of
$$r_i^*: H^2(S_i,\C)\rightarrow H^2(S_i,\C).$$
By recalling that $$16=\chi(\Fix(r))=\chi(\Fix(r_1\times r_2))=\chi(\Fix(r_i))^2$$ and by Lefschetz Fixed Point formula, we have
$$\Tr(r^*)=\Tr(r_1^*)+\Tr(r_2^*)=\chi(\Fix(r_1^*))-2+\chi(\Fix(r_2^*))-2=2(4-2)=4.$$
With the same method we obtain $\Tr((t^*)^2)=\Tr(r^*(t^*)^2)=4$.
We can write $t$ as $(t_1\times t_2)\circ \sigma$ where $\sigma$ is the  the permutation of the two copies of $S$. Hence $t^*$ will swap $H^2(S_1)$ and $H^2(S_2)$ in the sum $H^2(S_1)\oplus H^2(S_2)$ and this means that its trace is zero. In the same way we obtain $\Tr((t^*)^3)=\Tr(r^*t^*)=\Tr(r^*(t^*)^3)=0$.
Merging these results and recalling that $\chi(Y)=-32$, we obtain
$$h^{1,1}(Y/G)=\frac{12+4+4+4+0+0+0+0}{8}=3\quad \mbox{ and } \quad h^{1,2}(Y/G)=5$$
so the quotient has the following Hodge diamond
$$\xymatrix@R=8pt@C=6pt{
   &   &   & 1   \\
   &   & 0 &   & 0 \\
   & 0 &   & 3 &   & 0 \\
1  &   & 5 &   & 5 &   & 1 \\
   & 0 &   & 3 &   & 0 \\
   &   & 0 &   & 0 \\
   &   &   & 1
}$$
In particular, the height is $8$.


\subsection{$\P^2\times \P^2$ with maximal order $9$}\label{p2p2}

Let $(x_0:x_1:x_2)$ and $(y_0:y_1:y_2)$ be the projective coordinates on the two copies of $\P^2$ and set $a=\e^{2\pi i/3}$. Consider the automorphism of $\P^2\times \P^2:=X$ defined by
$$g:=(x_0 : a x_1 : a^2x_2)\times (y_0 : a y_1 : a^2y_2):=g_1\times g_2$$
and
$$h:=(x_1 : x_2 : x_0) \times (y_1 : y_2 : y_0):=h_1\times h_2.$$
It is easy to show that the group $G$ generated by $g$ and $h$ is isomorphic to $\Z_3\times \Z_3$.

\noindent Moreover, it is easy to see that
\begin{align*}
\Fix(G) = & (\Fix(g_1)\times \Fix(g_2))\cup (\Fix(h_1)\times \Fix(h_2))\cup \\
 &  (\Fix(g_1h_1)\times \Fix(g_2h_2))\cup (\Fix(g_1^2h_1)\times \Fix(g_2^2h_2))
\end{align*}
where
$$\Fix(g_i) = \left\lbrace(1: 0: 0), (0: 1: 0), (0: 0: 1)\right\rbrace $$
$$\Fix(h_i) = \left\lbrace(1: 1: 1), (1: a: a^2), (1: a^2: a)\right\rbrace $$
$$\Fix(g_ih_i) = \left\lbrace(1: 1: a^2), (1: a^2: 1), (a^2: 1: 1)\right\rbrace $$
$$\Fix(g_i^2h_i) = \left\lbrace(1: 1: a), (1: a: 1), (a: 1:1)\right\rbrace.$$
Consider the following global sections of $\O_{\P^2}(3)=-K_{\P^2}$:
\[ \begin{array}{cc}
e_{i,0} = x_0^3+a^{2i}x_1^3+a^ix_2^3 & e_{i,1} = x_0^2x_1+a^{2i}x_1^2x_2+a^ix_0x_2^2 \\
e_{i,2} = x_0x_1^2+a^{2i}x_1x_2^2+a^{i}x_0^2x_2 & e_{0} = x_0x_1x_2 \end{array} \]
Then $g^*(e_{i,j})=a^{j}e_{i,j}$, $h^*(e_{i,j})=a^{i}e_{i,j}$, $g^*(e_0)=h^*(e_0)=e_0$; hence $$\left\lbrace e_{0},e_{i,j}\right\rbrace_{0\leq i,j\leq 2}$$ is a basis of $H^0(\P^2,\O_{\P^2}(3))$ composed of eigenvectors of both $g^*$ and $h^*$.  Since
$$H^0(X,-K_X)\simeq H^0(\P^2,-K_{\P^2})\otimes  H^0(\P^2,-K_{\P^2}),$$
a basis for the space of invariant sections is given by $$\left\lbrace e_{i_1,j_1}\otimes e_{i_2,j_2}\right\rbrace_{i_1+i_2\equiv_3 0,j_1+j_2\equiv_3 0} \cup \left\lbrace e_{0}\otimes e_{0,0},e_{0,0}\otimes e_{0},e_{0}\otimes e_{0}\right\rbrace.$$
By direct computation,  we can show that the generic invariant section doesn't intersect $\Fix(G)$. Moreover, the system $|H^0(X,-K_X)^G|$ is base-point free. By Bertini's Theorem, the generic section is smooth. Hence there exists a Calabi-Yau threefold $Y$ embedded in $\P^2\times \P^2$ equipped with a free action of $G$.
\vspace{4mm}

\noindent The space $H^2(X,\Z)$ is free of rank two and is generated by $\pi_1^*H$ and $\pi_2^*H$ where $<H>= H^2(\P^2,\Z)$. Every automorphism of $\P^2$ fixes $H$, so $H^2(X,\C)^G= H^2(X,\C)$. This implies that the following is the Hodge diamond of $Y/G:$
$$\xymatrix@R=8pt@C=6pt{
   &   &   & 1   \\
   &   & 0 &   & 0 \\
   & 0 &   & 2 &   & 0 \\
1  &   & 11 &   & 11 &   & 1 \\
   & 0 &   & 2 &   & 0 \\
   &   & 0 &   & 0 \\
   &   &   & 1.
}$$
Its height is $13$. An element $g\in\Aut(\P^2\times\P^2)=(\Aut(\P^2)\times \Aut(\P^2)) \ltimes\Z_2$ of order $3$ has to be of the form $g=g_1\times g_2$ with $g_i\in\Aut(\P^2)$. This means that $H^{2}(X,Z)^{<g>}= H^{2}(X,\Z)$ and thus the minimum for $h(Y/G)$ is achieved by this example.


\subsection{$dP_5 \times dP_5$ with maximal order $5$}

Fix $P_1=(1:0:0)$, $P_2:=(0:1:0)$, $P_3:=(0:0:1)$ and $P_4:=(1:1:1)$ in ${\P}^2$. Let $S$ be the unique del Pezzo surface of degree five. It is well known that the automorphism group of $S$ is isomorphic to the symmetric group of order $120$. The sections of $\O_S(-K_{S})$ are the cubics through the points $P_i$ for $i=1,2,3,4$. Hence a basis of $H^0(dP_5, -K_{dP_5})$ can be taken to be
\[ \begin{array}{cc}
y_1:=x_0^2x_1-x_0x_1x_2 & y_2:=x_0^2x_2-x_0x_1x_2 \\
y_3:=x_1^2x_0-x_0x_1x_2 & y_4:=x_1^2x_2-x_0x_1x_2 \\
y_5:=x_2^2x_0-x_0x_1x_2 & y_6:=x_2^2x_1-x_0x_1x_2
\end{array} \]
where $x_0, x_1, x_2$ is a system of homogeneous coordinates on $\P^2$. Consider the following transformation $T$ on the projective plane, namely:
$$(x_0:x_1:x_2)\mapsto (x_0(x_0-x_2): x_0(x_0-x_1):(x_0-x_1)(x_0-x_2)).$$
\noindent It is easy to check that $T$ acts as automorphism of $S$ and its action on $H^0(S, -K_{S})$ is determined by
\[ \begin{array}{cc}
y_1 \mapsto y_1 & y_2 \mapsto y_1+y_5-y_2 \\
y_3 \mapsto y_2  & y_4 \mapsto y_2+y_3-y_1 \\
y_5 \mapsto -y_6+y_5-y_2 & y_6 \mapsto -y_4 - y_1 + y_3.
\end{array} \]
\noindent It's easy to see that the order of $T$ is five thus $G:=<T>$ is isomorphic to $\Z_5$. Let us now consider the action of $G$ diagonally on $X=S \times S$. We will use $x_0,x_1,x_2$ and $z_0,z_1,z_2$ as projective coordinates on the two $\P^2$'s we blow up to obtain the two copies of $S$. There is an action of $G$ on $H^0(X, - K_{X})$. Let $\omega$ be a primitive fifth root of unity. The space of invariants under this action is generated by the following polynomials, namely:
\[ \begin{array}{c}
f_1g_1=x_0x_1(x_0-x_2)z_0z_1(z_0-z_2), \\
 f_1g_2=x_0x_1(x_0-x_2)z_2(z_1-z_2)(z_0-z_1), \\
f_2g_1=x_2(x_1-x_2)(x_0-x_1)z_0z_1(z_0-z_2), \\
 f_2g_2=x_2(x_1-x_2)(x_0-x_1)z_2(z_1-z_2)(z_0-z_1), \\
h_1k_4, \quad h_2k_3, \quad h_3k_2, \quad h_4k_1,
\end{array}\]
where we set:
\[ \begin{array}{cl}
h_1 = & (1+\omega^2)y_1+(\omega^3-\omega^2)y_2+(-2-\omega-\omega^2-\omega^3)y_3+\\
& +y_4+\omega^2y_5-\omega y_6, \\
h_2 = & -(\omega+\omega^2+\omega^3)y_1+(1+2\omega+\omega^2+\omega^3)y_2+ \\
& +(\omega^3-1)y_3+y_4-(1+\omega+\omega^2+\omega^3)y_5-\omega^2 y_6, \\
h_3 = & (1+\omega)y_1+(-2\omega-1-\omega^2-\omega^3)y_2+(\omega^2-1)y_3+y_4+\\
& +\omega y_5-\omega^3 y_6, \\
h_4 = & y_1+(\omega^2-1)y_2+(\omega^3-1)y_3+(1+\omega^2+\omega)y_4-(\omega+\omega^2)y_5+\\
& -(\omega^2+\omega^3) y_6
\end{array} \]
and $k_i=h_i(z_0,z_1,z_2)$. It is easy to check that $h_i$ are eigenvectors with eigenvalue $\omega^i$ with respect to the action of $T$ on $H^0(S, - K_{S})$.

\noindent Let
\begin{equation}
\label{esse}
s:=A_1f_1g_1+A_2f_1g_2+A_3f_2g_1+A_4f_2g_2+A5h_1k_4+A_6h_2k_3+A_7h_3k_2+A_8h_4k_1,
\end{equation}
where $A_i$ are complex numbers not all of which are zero. For any choice of the $A_i$'s we get a section in $H^0(X , -K_X)$ which is invariant with respect to $G$.
\vspace{4mm}

\noindent The transformation $T$ acts with fixed points on $X$. They are given by $$(1:1/(1+\rho):1-\rho)$$ where $\rho$ satisfies the degree two equation $\rho^2 + \rho -1=0$ and thus there are four fixed points on $X$. Note that, by the Lefschetz Formula this is the least number of fixed points one could obtain. By a suitable choice of the $A_i$'s, we restrict to the locus $\Sigma$ such that the sections $s$ in \eqref{esse} do not pass through the fixed points above. It is easy to see that $\Sigma$ is not empty. For $s \in \Sigma$ the set of zeroes $Y=V(s)$ is thus invariant with respect to the free action of $G$ on it.
\vspace{4mm}

\noindent Now, we look for base points of the system above. First, we look for solutions on ${\mathbb P}^2 \times {\mathbb P}^2$ of the equations
$$f_1g_1 = f_1g_2 = f_2g_1= f_2g_2 = h_1k_4 = h_2k_3 = h_3k_2 = h_4k_1=0.$$
\ifVERSIONEPROLISSA
\else 
Next, we recall that $S$ is obtained from ${\mathbb P}^2$ by blowing-up the points $P_i$. After some computation we show that there are $20$ base points.
\vspace{4mm}
\fi 

\noindent For each of the base points we checked if they are smooth or not for the generic section. This is true if we restrict to an dense open set $\Omega$ of ${\mathbb P}^7$, where $\left\lbrace A_i\right\rbrace_{i=1..8}$ are interpreted as a homogeneous system of coordinates.
\ifVERSIONEPROLISSA
\else 
For example, let us take the point $$(((1:1:1),(1:1)),(1+\omega^2+\omega^3:1:\omega^2+\omega^3)),$$ that is the point whose projection on $S_1$ is the point $(1:1)$ on the exceptional divisor associated to $(1:1:1)$ and whose projection on $S_2$ is
$(1+\omega^2+\omega^3:1:\omega^2+\omega^3)$.
\fi 
We first make the substitution $x_0=w_0+w_1, x_1=w_1, x_2=w_1+w_2$, so the point $(1:1:1)$ is mapped to the point $(0:1:0)$. Next we work in the local chart where the second coordinate is nonzero. Let $((u,v),(l:m))$ be the coordinate on blow-up. Since $m\neq 1$, we can consider affine coordinates $v, l$ and, by the equation of the blow-up, $u=vl$. Thus, we evaluate all the polynomials $f_1g_1, f_1g_2, f_2g_1, f_2g_2, h_1k_4, h_2k_3, h_3k_2, h_1k_4$ at $w_0=vl, w_1=1, w_2=v$. We divide by $v$ and then take the derivatives with respect to $v,l, z_0, z_1, z_2$. These must be evaluated at $l=1, v=0$ and $z_0=1+\omega^2+\omega^3,  z_1=1, z_2=\omega^2+\omega^3.$ Doing so yields conditions on the $A_i$'s. These conditions define the equations of a closed set, the complement of which is the non-empty open set $\Omega$. The intersection of $\Omega$ and $\Sigma$ yield an open set which contains sections $s$ which are invariant with respect to $G$, do not pass through fixed points and such that $(s=0)$ is smooth at the base points. By Bertini's Theorem a generic element of $\Omega \cap \Sigma$ is smooth. This yields a Calabi-Yau manifold $Y/G$ with Euler characteristic $-10$. As in section \ref{p14} we compute $h^{1,1}(Y/G)$ using the Lefschetz Formula and we obtain $h^{1,1}(Y/G)=2$. Then $h^{2,1}(Y/G)=7$ and the Hodge diamond is the following one:
$$\xymatrix@R=8pt@C=6pt{
   &   &   & 1   \\
   &   & 0 &   & 0 \\
   & 0 &   & 2 &   & 0 \\
1  &   & 7 &   & 7 &   & 1 \\
   & 0 &   & 2 &   & 0 \\
   &   & 0 &   & 0 \\
   &   &   & 1
}$$
Note that $Y/G$ realize the minimum for the height.


\subsection{$\P^1\times \P^1\times dP_4$ with maximal order $4$}

Let us consider again the del Pezzo surface $S_2$ of degree $4$ embedded in $\P^4$ used in section \ref{SEC:dP4xdP4}. If we denote with $g_1$ and $h_1$ the automorphism of $S_1=\P^1\times \P^1$ such that
$$g_1((x_{10}:x_{11}),(x_{20}:x_{21}))=((x_{10}:-x_{11}),(x_{20}:-x_{21}))$$
and
$$h_1((x_{10}:x_{11}),(x_{20}:x_{21}))=((x_{11}:x_{10}),(x_{21}:x_{20})),$$
we obtain the relation $g_1^2=h_1^2=g_1h_1g_1^{-1}h_1^{-1}=\id$ that is $<g_1,h_1>\simeq \Z_2\oplus \Z_2$.
The same holds for the automorphism $g_2$ and $h_2$ of $S_2$ such that
$$g_2((y_0: y_1: y_2: y_3: y_4))=(y_0: y_1: -y_2: y_3: -y_4)$$
and
$$h_2((y_0: y_1: y_2: y_3: y_4))=(y_0: y_1: -y_2: -y_3: y_4).$$
Denote by $g=g_1\times g_2$ and $h=h_1\times h_2$; hence we have $G:=<g,h>\simeq \Z_2\oplus \Z_2$.
\vspace{4mm}

\noindent We recall (see Section \ref{SEC:dP4xdP4}) that if $a$ and $b$ are fixed roots of $2z^2+1+i$ and $2z^2+1-i$ then
$$\Fix(g_2)=\left\lbrace (1:\pm a : 0 :\pm b : 0) \right\rbrace$$
$$\Fix(h_2)=\left\lbrace (\pm a:\pm b : 0 : 0 : 1) \right\rbrace$$
and
$$\Fix(g_2h_2)=\left\lbrace(\pm b:1:\pm a: 0: 0)\right\rbrace.$$
It is easy to see that $|\Fix(\alpha)|=4$ for each $\alpha\in<g_1,h_1>\setminus \left\lbrace \id \right\rbrace$ and, consequently, that $|\Fix(G)|=48$.
\vspace{4mm}

\ifVERSIONEPROLISSA
\else 
\fi 
\noindent Analogously to the previous cases, we can conclude that there exists a smooth Calabi-Yau threefold $Y\subset X$ and a group $G\simeq \Z_2\oplus \Z_2$ acting freely on it. The quotient has the following Hodge diamond
$$\xymatrix@R=8pt@C=6pt{
   &   &   & 1   \\
   &   & 0 &   & 0 \\
   & 0 &   & 5 &   & 0 \\
1  &   & 13 &   & 13 &   & 1 \\
   & 0 &   & 5 &   & 0 \\
   &   & 0 &   & 0 \\
   &   &   & 1.
}$$
Hence the height of the quotient is $18$.


\ifCUTSECTION
\else 
\subsection{Other similar examples}\label{other}

For brevity we don't treat explicitly some examples. These are some threefolds  in $\P^2\times dP_6, \P^2\times dP_3, (\P^1\times \P^1)\times dP_6$ and $dP_6\times dP_4$.
The threefolds in $\P^2\times dP_6$ and in $\P^2\times dP_3$ admit a free action of $\Z_3$ (in both cases $M(S_1,S_2)=3$). The quotients have Hodge diamonds respectively:
$$\xymatrix@R=8pt@C=6pt{
   &   &   & 1   \\
   &   & 0 &   & 0 \\
   & 0 &   & 3 &   & 0 \\
1  &   & 21 &   & 21 &   & 1 & \quad\mbox{and }\quad\\
   & 0 &   & 3 &   & 0 \\
   &   & 0 &   & 0 \\
   &   &   & 1
}
\xymatrix@R=8pt@C=6pt{
   &   &   & 1   \\
   &   & 0 &   & 0 \\
   & 0 &   & 4 &   & 0 \\
1  &   & 13 &   & 13 &   & 1 \\
   & 0 &   & 4 &   & 0 \\
   &   & 0 &   & 0 \\
   &   &   & 1
}$$
 These are threefolds with minimal height. The threefolds in $(\P^1\times\P^1)\times dP_6$ and in $dP_4\times dP_6$ admit a free action of $\Z_2$ (again this hits the maximum because $M(S_1,S_2)=2$ for these two cases). The Hodge diamonds are
$$\xymatrix@R=8pt@C=6pt{
   &   &   & 1   \\
   &   & 0 &   & 0 \\
   & 0 &   & 5 &   & 0 \\
1  &   & 29 &   & 29 &   & 1 & \quad\mbox{and }\quad\\
   & 0 &   & 5 &   & 0 \\
   &   & 0 &   & 0 \\
   &   &   & 1
}
\xymatrix@R=8pt@C=6pt{
   &   &   & 1   \\
   &   & 0 &   & 0 \\
   & 0 &   & 7 &   & 0 \\
1  &   & 19 &   & 19 &   & 1 .\\
   & 0 &   & 7 &   & 0 \\
   &   & 0 &   & 0 \\
   &   &   & 1
}$$

\fi 


\section{Results of non-existence}\label{results}

In this section we present some results of non-existence. In particular, we show that there are cases for which $M(S_1,S_2)>1$ but a group $G$ that fulfills our requests doesn't exist.


\subsection{$dP_8\times S$, with $S\in\left\lbrace\P^1\times \P^1,dP_8,dP_6,dP_4,dP_2\right\rbrace$}

We will show that in these cases $m(S_1,S_2,Y)=1$ for all $Y$. The key points are Corollary \ref{COR:Fix} and some structural results on $\Aut(dP_8)$.

\begin{lem}
If $S$ is a del Pezzo surface and $g\in\Aut(S)$ is such that $o(g)=p$ is prime, then $g$ has a fixed point.
\end{lem}

\begin{proof}
Every del Pezzo surface $S$ is a rational surface. Suppose that the fixed locus of $g$ is empty. Recall that $p$ is prime. Let $G:=<g>$ be the group generated by $g$. Then $\Fix(G)$ is empty. In fact, for every $n\not\equiv_p 0$ there exists $m$ such that $nm\equiv_p 1$; this implies
$$\Fix(g^n)\subset\Fix((g^n)^m)=\Fix(g).$$
Therefore $R:=S/G$ is a smooth surface and $R$ is rational. In particular $\Pi_1(R)=\left\lbrace\id\right\rbrace$. But this is not possible because $S$ is simply connected, so $\Pi_1(R)\simeq G\not\simeq\left\lbrace\id\right\rbrace$. Hence, $g$ must have at least one fixed point.
\end{proof}

\begin{cor}
\label{COR:Fix}
For every finite subgroup $G$ of $\Aut(S)$, $|\Fix(G)|>0$.
\end{cor}

\noindent By \cite{Dolgachev}, every automorphism of a del Pezzo surface $S$ of degree $8$ comes from an automorphism of $\P^2$ that fixes the point $R$ such that $S=\Bl_{\left\lbrace R\right\rbrace}\P^2$.
Suppose $S\neq dP_8$. Then we search for a group $G\leq\Aut(dP_8)\times \Aut(S)$. We are interested in the cases $S\in\left\lbrace\P^1\times \P^1,dP_6,dP_4,dP_2\right\rbrace$ for which $M(dP_8,S)$ is respectively $16,2,4$ and $2$. It is then enough to show that there are not groups of order $2$ whose action is free on $Y$. Let $g=(g_1,g_2)$ be an involution. By Corollary \ref{COR:Fix} there exists a fixed point $P$ of $g_2$. The automorphism $g_1$ comes from an involution of $\P^2$, hence it has a line $L$ of fixed points, therefore $L\times \left\lbrace P\right\rbrace$ is a line of fixed points for $g$.
\vspace{4mm}

\noindent If $S=dP_8$, then $\Aut(dP_8^{\times 2})=\Aut(dP_8)^{\times 2}\ltimes\Z_2$. Let $G=<g>$ where $g=(g_1,g_2)$. Using the same result as above, we will have a surface of fixed points. Then, it suffices to analyze the case $g=(g_1, g_2)\circ \tau$, where $\tau$ is the involution that switches the two copies of $dP_8$. Then, by changing projective coordinates, we can assume that
$$(g_1,g_2)=\left(\left[\begin{array}{ccc}
 a & 0 & 0\\
 0 & b & 0 \\
 0 & 0& 1\end{array}\right],\left[\begin{array}{ccc}
 a^{-1} & 0 & 0\\
 0 & b^{-1} & 0 \\
 0 & 0& 1\end{array}\right]\right)$$
for some $a,b\in\C^*$. It is easy to see that $((ax:by:0),(x:y:0))$ is a line of fixed points.
\vspace{4mm}

\noindent In conclusion, we have shown that $m(dP_8,S,Y)=1$ for a del Pezzo surface $S$ (here we have checked all the cases for which $M(dP_8,S)\neq1$) and for all $Y$ Calabi-Yau embedded in $dP_8\times S$.


\subsection{$dP_7\times dP_7$ with estimated maximal order $7$}

There is only one del Pezzo surface $S$ of degree $7$. It is given as the blow-up of $\P^2$ in $P_0=(1:0:0)$ and $P_1=(0:1:0)$. We will show that there does not exist a section $s$ of $-K_{S\times S}$ such that $g^*s=c s$ for some $c\in\C^*$ and $ g\in\Aut(S\times S)$ of order $7$ which doesn't intersect the fixed locus of $<g>$.
\vspace{4mm}

\noindent By \cite{Dolgachev}, every automorphism of a del Pezzo surface of degree $7$ comes from an element of $PGL(3)$ fixing the set $\left\lbrace P_0, P_1\right\rbrace$. Thus, we have
$$\Aut(S)\simeq \left<\left[\begin{array}{ccc}
1 & 0 & b \\
0 & a & c \\
0 & 0 & d \end{array}\right],\left[\begin{array}{ccc}
0 & 1 & 0 \\
1 & 0 & 0 \\
0 & 0 & 1 \end{array}\right]\right>.$$
 Recall that $\Aut(S\times S)=\Aut(S)^{\times 2}\ltimes \Z_2$. Since we need $g$ of order $7$, we have to choose an element of the form $g=(g_1, g_2)$, where $g_i\in\Aut(S)$ and
$$g_i=\left[\begin{array}{ccc}
1 & 0 & b_i \\
0 & a_i & c_i \\
0 & 0 & d_i \end{array}\right].$$
After a change of projective coordinates that fixes the points $P_0$ and $P_1$, we may assume $b_i=c_i=0$ so that $g_i$ is in diagonal form. The condition $o(g)=7$ gives $a_i^7=d_i^7=1$. Since we need a finite number of fixed points, we must impose $a_i\neq1\neq d_i$ and $a_i\neq d_i$.
\vspace{4mm}

\noindent In conclusion, we can take $g$ of the form
$$\left(\left[\begin{array}{ccc}
1 & 0 & 0 \\
0 & \lambda^{m_1} & 0 \\
0 & 0 & \lambda^{n_1} \end{array}\right]\times \left[\begin{array}{ccc}
1 & 0 & 0 \\
0 & \lambda^{m_2} & 0 \\
0 & 0 & \lambda^{n_2} \end{array}\right]\right)$$
where $\lambda=\e^{2\pi i/7}$ and $0\neq n_i,m_i$ and $n_i\neq m_i$.
\vspace{4mm}

\noindent The fixed points of $g_i$ as an automorphism of $\P^2$ are $P_0,P_1$ and $P_2$, whereas the fixed points of $g_i$ as an automorphism of $S$ are $$\left\lbrace (P_0,Q),(P_1,Q),P_2\, |\, Q\in \left\lbrace(1:0),(0:1)\right\rbrace\right\rbrace.$$
Here, for example, with $((0:1:0),(1:0))$ we mean the point $(1:0)$ on the exceptional divisor $E_1=\pi^{-1}(P_1)$, where we use the standard local description of $S$ in a neighbourhood of $E_1$ as the surface of $\C^2\times\P^1$ such that $um=vl$ with $\left\lbrace((0,0),(l:m))\right\rbrace= E_1$. Hence, in total, $G:=<g>$ has $25$ fixed points.
\vspace{4mm}

\noindent We blow up $\P^2$ in $P_0$ and $P_1$. Then, the following isomorphism holds:
$$H^0(S,-K_2)\simeq <x_2^3,x_0^2x_1,x_0^2x_2,x_0x_1^2,x_0x_2^2,x_1^2x_2,x_1x_2^2,x_0x_1x_2>.$$
The correspondence is given by taking the strict transform of a polynomial see as a global section of $\O_{\P^2\times \P^2}(3,3)$.
We call $e_i$ the elements of the base on the first del Pezzo surface and $f_i$ the elements of the base on the second one so that, by the K\"unneth formula, we obtain
$$H^0(S\times S,-K_{S\times S})\simeq <e_i\otimes f_j>.$$
Suppose that $s$ is an eigenvector of $H^0(S\times S,-K_{S\times S})$ and that $s(P)\neq 0$ for all $P$ fixed points of $G$. Then, for example, $$s(((1:0:0),(1:0)),((1:0:0),(1:0)))\neq 0$$ if and only if $s$ belongs to the eigenspace of $x_0^2x_1y_0^2y_1$ and
$$s(((1:0:0),(1:0)),((1:0:0),(0:1)))\neq 0$$ if and only if $s$ is in the eigenspace of $x_0^2x_1y_0^2y_2$.
But these two eigenvectors have corresponding eigenvalues $\lambda^{m_1+m_2}$ and $\lambda^{m_1+n_2}$ and these numbers are different if and only if $m_2\neq n_2$, which it is true by hypotesis. This means that $s$ must be zero and we have a contradiction.
\vspace{4mm}

\noindent Albeit  $M(dP_7,dP_7)=7$, this shows that an automorphism of $S\times S$ with finite order cannot act freely on a smooth section of $-K_{S\times S}$.

\subsection{$dP_6\times dP_3$ with estimated maximal order $9$}

In this case recall that $M(dP_3,dP_6)=9$. Nonetheless, the maximum order of $G$ to have a free action on a Calabi-Yau threefold $Y$ embedded in $X$ is $3$. We will also give an example for which $m(dP_6,dP_3,Y)=3$.
\vspace{4mm}

\noindent Suppose that $G\leq\Aut(dP_6)\times \Aut(dP_3)$ has order $9$. Then either $G\simeq \Z_9$ or $G\simeq \Z_3\times \Z_3$. First, we will show that if $G\simeq \Z_9$ then $G$ must have a fixed curve and so it can't satisfy our assumption on $G$. Next, we will deal with the case $G\simeq \Z_3\times \Z_3$. We'll first find all the groups whose fixed locus is finite. Essentially, this will be done by projecting $G$ on $\Aut(dP_6)$ and $\Aut(dP_3)$ so that the projections $G_1$ and $G_2$ satisfy $G_1\simeq G_2\simeq G\simeq \Z_3\times \Z_3$. There is only one useful choice for $G_2=<g_2,h_2>$ whereas there are infinitely many possibilities for $G_1$, which  are parametrized by $(\C^*)^2$. Once we fix $G_1:=<u,v>$, we will consider all the possible $G'$s such that the projection of $G$ on $\Aut(dP_3)$ and $\Aut(dP_6)$ are $G_1$ and $G_2$, respectively. This will be done by choosing all the possible pairs $(g_1,h_1)$, not necessarily equal to $(u,v)$, that generate $G_1$. We thus  consider the group $G:=<g,h>$, where $g=g_1\times g_2$ and $h=h_1\times h_2$. For every case we have checked that all the sections of $H^0(X,-K_X)$ that are eigenvectors of both $g^*$ and $h^*$ are zero on a fixed point of the group $G$ (we will show an explicit calculation for one of the cases).
\vspace{4mm}

\noindent Suppose that $G\simeq \Z_9$ and consider its projection $G_1$ on $\Aut(dP_3)$. Necessarily, $G_1\simeq G$. On the contrary, if $G=<g_1\times g_2>$ with $g_1^3=\id$, $G$ would have infinitely many fixed points. Hence $G_1$ has to be a group isomorphic to $\Z_9$ in $\Aut(dP_3)$. If $S$ is a smooth cubic surface in $\P^3$ and if $g_1\in \Aut(S)$ has order $9$ then, by \cite{Dolgachev}, there exist a projective automorphism of $\P^3$ such that $$(S,g_1)=\left(V(x_0^3+x_2^2x_0+x_1^2x_2+x_0^2x_1),\left[\begin{array}{cccc}
1 & 0 & 0 & 0\\
0 & a^4 & 0 & 0\\
0 & 0 & a & 0  \\
0 & 0 & 0 & a^7\end{array}\right]\right)$$
where $a$ satisfies $a^3\neq 1 = a^9$. On the other hand, we have
$$g_1^3=\left[\begin{array}{cccc}
1 & 0 & 0 & 0\\
0 & a^3 & 0 & 0\\
0 & 0 & a^3 & 0  \\
0 & 0 & 0 & a^3\end{array}\right].$$
Hence $\Fix(<g_1>)$ contains a curve $C$. This means that, by Corollary \ref{COR:Fix}, we have a fixed curve in $\Fix(G)$, which contradicts our assumptions.
\vspace{4mm}

\noindent Suppose, now, that $G\simeq \Z_3\times \Z_3\leq \Aut(dP_6)\times \Aut(dP_3)$ and consider the projection $G_2$ on $\Aut(dP_3)$ so that $G_2\simeq G$. Fix two generators $g_2,h_2$ of $G_2$ and consider $dP_3=V(f)\subset \P^3$. By \cite{Dolgachev}, if $V(f)$ is a smooth cubic and $\tilde{G}\simeq \Z_3\times \Z_3\leq \Aut(V(f)),$ we can change coordinates to obtain $f=\sum y_i^3$. In this case $\Aut(V(f))\simeq \Z_3^3 \ltimes S_4$, where each $\Z_3$ acts as multiplication of a variable by $a^k$ (we write the elements in $\Z_3^3$ as $(1,a^{k_1},a^{k_2},a^{k_3})$) and $S_4=\Sym(0,1,2,3)$ is generated by the permutation of the variables. By requiring $|\Fix(G_2)| < \infty$ we obtain $G_2\leq \Z_3^3$. There is only one group isomorphic to $G_2$ in $\Z_3^3$ that has a finite number of fixed points on $V(f)$ and it is $<g_2,h_2>$ where $g_2=(1,1,a,a^2)$ and $h_2=(1,a,a^2,a^2)$. We call $V_{i,j}^{(2)}$ the maximal subspace of $H^0(dP_3,-K_{dP_3})$ such that $g_2^*(s)=a^i s$ and $h_2^*(s)=a^j s$ for every $s\in V_{i,j}^{(2)}$. This vector space is the intersection of the eigenspaces $\Lambda_{a^i}$ of $g_2$ and $\Lambda_{a_j}'$ of $h_2$ relative to $a^j$. The following table summarizes the situation providing generators for these spaces.
\vspace{4mm}

\[ \begin{array}{c||c|c|c}
g_2 \backslash h_2 & \Lambda_{1}' & \Lambda_{a}' & \Lambda_{a^2}'  \\ \hline\hline
\Lambda_{1}     & x_0 & & \\ \hline
\Lambda_{a}     & x_1 & & \\ \hline
\Lambda_{a^2} & & x_2 & x_3
 \end{array} \]

\noindent Now, consider the projection $G_1$ of $G$ on $\Aut(dP_6)=(S_3\times \Z_2)\ltimes (\C^*)^2$. Any element of order $3$ can be written in the form $\diag(1,b,c)\circ (123)^k$ for some fixed $b,c\in\C^*$ and $k=0,1,2$. Easy arguments show that $G_1$ cannot satisfy $G_1\leq (\C^*)^2$ (if it happens, one has $|\Fix(G_1)|=\infty$) and that $G_1$ has exactly two non-trivial elements in $(\C^*)^2$. These are $\diag(1,a,a^2)$ and its inverse. Moreover, these two elements commute with every element of the form $(1,b, c)\circ (123)^k$, thus every subgroup of $\Aut(dP_6)$ isomorphic to $\Z_3\times \Z_3$ and with a finite number of fixed points can be written in the form  $<u,v>$ where
$$u=\diag(1:a:a^2)\quad\mbox{ and }\quad v=\diag(1:b:c)\circ (123)$$
for some fixed $b,c\in \C^*$. We define $d$ to be a fixed third root of $bc$. Set
\begin{align*}
F_{0} = & x_{10}x_{20},\\
F_{1} = & x_{10}x_{21}+\frac{1}{b}x_{11}x_{22}+\frac{1}{c}x_{12}x_{20},\\
F_{2} = & x_{10}x_{22}+\frac{1}{c}x_{11}x_{20}+\frac{b}{c}x_{12}x_{20},\\
F_{3} = & x_{10}x_{21}+\frac{a^2}{b}x_{11}x_{22}+\frac{a}{c}x_{12}x_{20},\\
F_{4} = & x_{10}x_{22}+\frac{a^2}{c}x_{11}x_{20}+\frac{ab}{c}x_{12}x_{20},\\
F_{5} = & x_{10}x_{21}+\frac{a}{b}x_{11}x_{22}+\frac{a^2}{c}x_{12}x_{20},\\
F_{6} = & x_{10}x_{22}+\frac{a}{c}x_{11}x_{20}+\frac{a^2b}{c}x_{12}x_{20}.
\end{align*}

\noindent Then $F_j$ is an eigenvector of both $u$ and $v$ and the corresponding eigenvalues are the ones in the following table:
\[ \begin{array}{c||c|c|c}
u \backslash v & \Lambda_{1} & \Lambda_{a} & \Lambda_{a^2}  \\ \hline\hline
\Lambda_{1}     & F_0 & F_2 & F_1 \\ \hline
\Lambda_{a}     &        & F_4 & F_3 \\ \hline
\Lambda_{a^2} &       &  F_6 & F_5
 \end{array} \]
This shows that $\left\lbrace F_j\right\rbrace$ form a base for $H^0(dP_6,-K_{dP_6})$. The following are the fixed points of the elements of $G_1$ and $G_2$:
\[ \begin{array}{c||c}
\mbox{Element} & \mbox{ Fixed points } (k=0,1,2) \\ \hline\hline
 & ((1:0:0),(0:1:0)),((1:0:0),(0:0:1)),\\
u, u^2 & ((0:1:0),(1:0:0)), ((0:1:0),(0:0:1)), \\
 & ((0:0:1),(1:0:0)),((0:0:1),(0:1:0)) \\ \hline
v, v^2 & ((1:da^k:\frac{(da^k)^2}{b},(1:\frac{1}{da^k}:\frac{b}{(da^k)^2}) \\ \hline
uv, u^2v^2 & ((1:da^k:\frac{(da^k)^2}{ba},(1:\frac{1}{da^k}:\frac{ba}{(da^k)^2}) \\ \hline
u^2v, uv^2 & ((1:da^k:\frac{(da^k)^2}{ba^2},(1:\frac{1}{da^k}:\frac{ba^2}{(da^k)^2})
 \end{array} \]
\[ \begin{array}{c||c}
\mbox{Element} & \mbox{ Fixed points } (k=0,1,2) \\ \hline\hline
g_2, g_2^2 & (1:-a^k:0:0) \\ \hline
h_2, h_2^2 & (0:0:1:-a^k) \\ \hline
g_2h_2, g_2^2h_2^2 & (1:0:-a^k:0),(0:1:0:-a^k) \\ \hline
g_2h_2^2, g_2^2h_2 & (1:0:0:-a^k),(0:1:-a^k:0)
 \end{array} \]

\noindent Suppose $g_1=u$. Let $h_1$ be any element of $G_1$ such that $G_1=<g_1,h_1>$ and denote $Q_1=((1:0:0),(0:1:0))$ and $Q_2=((1:0:0),(0:0:1))$. Then
$$P_1:=((1:0:0),(0:1:0),(1:-1:0:0))$$
and
$$P_2:=((1:0:0),(0:0:1),(1:-1:0:0))$$
are fixed points of $g=g_1\times g_2$. Suppose that
$$s=\sum_{i,j}a_{i,j}F_i y_j$$
is a section such that $g^*(s)=a^{k_1}s$ and that $s(P_j)\neq 0$. Then
$$s(P_1)=\sum_{i=2,4,6}(a_{i,0}-a_{i,1})F_i(Q_1)\neq 0$$
and
$$s(P_2)=\sum_{i=1,3,5}(a_{i,0}-a_{i,1})F_i(Q_2)\neq 0.$$
This means that at least one between $x_iF_j$ with $i=0,1$ and $j=2,4,6$ has a non zero coefficient and the same is true for $x_iF_j$ with $i=0,1$ and $j=1,3,5$. But, if $i=0,1$, $g^*(x_iF_j)=a^2 x_iF_j$ if $j=2,4,6$  and  $g^*(x_iF_j)=a x_iF_j$ if $j=1,3,5$. Then each eigenvector of $g$ is zero if evaluated in $P_1$ or in $P_2$.
\vspace{4mm}

\noindent The same result is true for every other case: we have checked that, for every $b,c\in (\C^*)$, for every choice of $g_1,h_1$ generators of $G_1=<u,v>$, every section of $H^0(X,-K_X)$ that is an eigenvector of both $g$ and $h$ where $g=g_1\times g_2$ and $h=h_1\times h_2$ is zero on at least one fixed point of $G=<g,h>$. In conclusion the restriction of the action of a group $G\leq \Aut(dP_6)\times \Aut(dP_3)$ of order $9$ to a Calabi-Yau threefold $Y\subset dP_6\times dP_3$ cannot be free. Hence $m(dP_6,dP_3,Y)<M(S_1,S_2)=9$ for every $Y$.
\vspace{4mm}

\noindent We have obtained $m(dP_6,dP_3,Y)\leq 3$ for all $Y$. We now give an example such that $m(dP_6,dP_3,Y)=3$. Take $dP_3$ to be the Fermat surface in $\P^3$. Call $g_1$ the automorphism of $dP_6$ such that $x_{i,j}\mapsto x_{i,j+1}$
and $g_2$ the authomorphism
$$\left[\begin{array}{cccc}
1 & 0 & 0 & 0\\
0 & 1 & 0 & 0\\
0 & 0 & \omega & 0 \\
0 & 0 & 0 & \omega^2\end{array}\right]$$
of $dP_3$. Notice that the minimum for the number of fixed points for an automorphism of order $3$ in $\Aut(dP_6)\times \Aut(dP_3)$ is achieved by $g=g_1\times g_2$.
\ifVERSIONEPROLISSA
\else 
The dimension of $H^0(X,-K_X)^G$, where $G=<g>$, is $10$.
\fi 
\noindent It can be shown that the base locus for $|H^0(X,-K_X)^G|$ has only $9$ points and that these are
$$((1:\omega^i:\omega^{2i}),(1:\omega^{2i}:\omega^{i}),(0:0:-\omega^j:1))$$
with $0\leq i,j\leq 2$. By direct inspection, the generic invariant section $s$ is smooth at these points and does not intersect the fixed locus, so, by Bertini's Theorem, there exists a Calabi-Yau $Y$ embedded in $dP_6\times dP_3$ and a group $G\simeq \Z_3$ acting freely on $Y$. The Hodge diamond for $Y/G$ is
$$\xymatrix@R=8pt@C=6pt{
   &   &   & 1   \\
   &   & 0 &   & 0 \\
   & 0 &   & 5 &   & 0 \\
1  &   & 11 &   & 11 &   & 1 \\
   & 0 &   & 5 &   & 0 \\
   &   & 0 &   & 0 \\
   &   &   & 1
}$$
and it's height is $16$, that is the minimum for the height.


\section{On the Relation between $\Aut(S_1)\times \Aut(S_2)$ and $\Aut(S_1\times S_2)$}\label{relation}

Let $X$ be a projective complex manifold. We will denote by $\NE(X)$ the cone of effective curves of $X$. An extremal subcone $V$ of $\NE(X)$ is a closed convex cone such that for every $v,w\in \NE(X)$ if $v+w\in V$ then $v,w\in V$. An extremal ray is an extremal subcone of dimension $1$. For every $D$ divisor on $X$ a subcone $V\subset \NE(X)$ is said to be $D-$negative if for every $v\in V$ one has $v\cdot D <0$. The Contraction Theorem says that for every extremal $K_X$-negative subcone $V$ of $\NE(X)$ the contraction $c_V$ of $V$ is well defined, that is to say, a morphism $c_V: X\rightarrow W$ with connected fibers such that  $W$ is a normal variety. Moreover, a curve in $X$ is contracted if and only if is numerically equivalent to a curve in $V$ and the Picard number $\rho(W)$ is equal to $\rho(X)-\Dim(<V>)$. For a morphism $f$ we recall that $\NE(f)$ is given by the intersection $\NE(X) \cap \ker(f_*)$, where $f_*$ is the map induced by $f$ on the vector space spanned by $\NE(X)$.
\vspace{4mm}

\noindent If $\phi\in\Aut(S_1\times S_2)$ we will write $\phi(x,y)=(\phi_1(x,y),\phi_2(x,y))$ where $\phi_i=\pi_i\circ \phi$ where $\pi_i$ is the projection of $S_1\times S_2$ on $S_i$.

\begin{lem}
\label{LEM:Split}
Let $S_1$ and $S_2$ be two del Pezzo surfaces and let $\phi\in \Aut(S_1\times S_2)$ Let $\pi_i$ be the projection from $S_1 \times S_2$ onto the $i$-th factor $S_i$ for $i=1,2$.
If $\phi_*(\NE(\pi_i))=\NE(\pi_i)$, then $\phi(x,y)=(\phi_1(x),\phi_2(y))$ where $\phi_i\in\Aut(S_i)$.
If $\phi_*$ switches the cones $\NE(\pi_1)$ and $\NE(\pi_2)$, then $S_1=S_2$ and $\phi(x,y)=(\phi_1(y),\phi_2(x))$ with $\phi_1\in\Biol(S_2,S_1)$ and $\phi_2\in\Biol(S_1,S_2)$.
\end{lem}
\begin{proof}
Assume $\phi_*(\NE(\pi_i))=\NE(\pi_i)$. Fix $x_1,x_2\in S_1$ and take two distinct irreducible curves $C_1$ and $C_2$ on $S_1$ whose intersection is non empty and such that $x_i\in C_i$. We have
$$\phi(C_i\times y)=D_i\times y_i$$ because the image of $C_i\times y$ is a curve that is numerically equivalent to a curve in $\NE(\pi_2)$. But $C_1\times y$ and $C_2\times y$ are two curves with nonempty intersection so their images have nonempty intersection. In particular $y_1=y_2$ and this implies that $\phi_2(x,y)=\phi_2(y)$. The same argument works with the first component ($\phi_1(x,y)=\phi_1(x)$) and with $\phi^{-1}$ meaning that $\phi_i$ is an automorphism of $S_i$.
\vspace{4mm}

\noindent With the same method, if $\phi_*$ switches the two cones, one has $$\phi(x,y)=(\phi_1(y),\phi_2(x))$$ and that $\phi_i$ are biholomorphism thus $S_1=S_2$.
\end{proof}

\begin{lem}
Let $S_1$ and $S_2$ be two del Pezzo surfaces such that $\rho(S_1),\rho(S_2)\geq 3$. If $\rho(S_1)\neq\rho(S_2)$ then $$\Aut(S_1\times S_2)=\Aut(S_1)\times\Aut(S_2).$$
The same holds if $\rho(S_1)=(S_2)$ and $S_1\neq S_2$. Instead, if $S_1=S_2$ one has
$$\Aut(S_1\times S_2)=(\Aut(S_1)\times \Aut(S_2))\ltimes \Z_2.$$
\end{lem}

\begin{proof}
Call $X$ the product $S_1\times S_2$. Then $X$ is a Fano fourfold and $$\NE(X)=\NE(X)\cap \NE(\pi_{1,*})+\NE(X)\cap \NE(\pi_{2,*}).$$ In particular, every extremal ray of $X$ is generated by a curve of the type $P_1\times E_2$ or $E_1\times P_2$, where $E_i$ is a $(-1)-$curve on $S_i$.
Observe that the image $V'$ of an extremal subcone $V$ by an automorphism $\phi$ is again an extremal subcone. In fact, if $v+w\in V'$ for some $v,w\in \NE(X)$ then $\phi_*^{-1}(v)$ and $\phi_*^{-1}(w)$ are effective curves such that $\phi_*^{-1}(v)+\phi_*^{-1}(w)=\phi_*^{-1}(v+w)\in V$.  But if  $V$ is extremal both $\phi_*^{-1}(v)$ and $\phi_*^{-1}(w)$ are in $V$. This implies that $v$ and $w$ are in $V'$, so $V'$ also is extremal.
This implies that $\phi$ induces a permutation of the extremal rays of $X$.
\vspace{4mm}

\noindent Suppose that there exists an extremal curve $E\times P_1$ such that $\phi_* (E_1\times P_2)=P_1\times E_2$. Then $\phi_*$ maps the extremal ray $V:=[E_1\times P_2]$ to the extremal ray $V':=[P_1\times E_2]$.  The contractions $c_V$ and $c_{V'}$ associated to the extremal subcones $V$ and $V'$ are respectively $p_1\times \id$ and $\id\times p_2$, where $p_i:S_i\rightarrow \hat{S}_i$ are the blow up with exceptional divisor $E_i$. Observe that $\hat{S}_i$ is smooth and that the fibers of $c_V$ and $c_{V'}$ have dimension $0$ or $1$ and are connected. By construction a curve $C$ is contracted by $c_{V}$ if and only $\phi_* C$ is contracted by $c_{V'}$. These two facts imply that the map $f:\hat{S}_1\times S_2 \rightarrow S_1\times \hat{S}_2$ such that $f(P)=(c_{V'}\circ \phi)(c_{V}^{-1}(P))$ is well defined.
$$\xymatrix{\ar @{} [dr] |{\circlearrowleft}
S_1\times S_2 \ar[d]_{c_{V}} \ar[r]^{\phi} & S_1\times S_2 \ar[d]^{c_{V'}} \\
\hat{S}_1\times S_2 \ar[r]_f & S_1\times \hat{S}_2 }$$
Let's see that the map $f$ is injective. Call $Q_i$ the point of $\hat{S}_i$ such that $p_i^{-1}(Q_i)=E_i$. If $f(Q_1\times R_1)=f(Q_1\times R_2)$ with $R_1\neq R_2$ then, to calculate the image of $Q_1\times R_i$ we obtain first two disjoint curves in $S_1\times S_2$ of the form $E_1\times R_i$. Then these two are sent to two disjoint curves of the form $T_i\times E_2$ by $\phi$ and, at last, contracted to the same point by $c_{V'}$. This implies that the fiber of this point with respect to $c_{V'}$ contains two disjoint curves and, being connected, has to be at least of dimension $2$. But we have seen that every fiber has dimension at most $1$, so necessarily $R_1=R_2$.
By construction $f$ is also surjective and so it is a bijective map.
\vspace{4mm}

\noindent The map $f$ is a morphism because it is everywhere well defined and it is holomorphic outside $Q_1\times S^2$ that has codimension $2$ in $\hat{S}_1\times S_2$. Hence, by Hartogs' Theorem, it is holomorphic on $\hat{S}_1\times S_2$. This is enought to conclude that $f$ is an isomorphism. This implies
$$\chi(\hat{S}_1\times S_2)=\chi(S_1\times \hat{S}_2);$$
but $\chi(\hat{S_i})=\chi(S_i)-1$ because $b_1(\hat{S}_i)=b_1(S_1)-1$ and hence, by the multiplicativity of $\chi$,  we have
$$(\chi(S_1)-1)\chi(S_2)=\chi(S_1)(\chi(S_2)-1)$$
and $\chi(S_1)=\chi(S_2)$. But this contraddicts the hypothesis $\rho(S_1)\neq\rho(S_2)$; hence the image of $E\times P_1$ by $\phi_*$ has to be of the same type. This implies that $\phi_*\NE(\pi_j)=\NE(\pi_j)$ and this is sufficient to conclude that $\phi$ can be written as a product of two automorphisms by
Lemma \ref{LEM:Split}.
\vspace{4mm}

\noindent Suppose, now, that $\rho(S_1)=\rho(S_2)\geq 3$. Fix a blow-up model for $S_i$. Then the $(-1)-$curves on $S_i$ are either $E_{ij}$, and are contracted to points by the model, or are sent to curves (lines, conics ($\rho(S_i)\geq 5$) and cubics ($\rho(S_i)\geq 7$)). If, for all $j$, the image of $E_{1j}\times P$ belongs to $[Q\times E]$ for some $(-1)-$curve $E$ that depends on $j$, then the same holds true for the other exceptional curves of the same type: $\phi(E_1\times P)\in [Q\times E]$ for some $E$ depending on $E_1$. Thus, saying that there exist two exceptional curves $E_i\times P$ such that $\phi(E_1\times P)\in [Q\times E]$ and $\phi(E_1\times P)\in [E'\times Q]$ is equivalent to requiring that there are two indices (for examples $j=1$ and $j=2$) such that
$$\phi(E_{11}\times P)\in [Q\times E_2]\mbox{ and }\phi(E_{12}\times P)\in [E_1\times Q].$$
Suppose, then, that this could happen. Then, as in the previous case, we can construct a commutative diagram
$$\xymatrix{\ar @{} [dr] |{\circlearrowleft}
S_1\times S_2 \ar[d]_{c_{V}} \ar[r]^{\phi} & S_1\times S_2 \ar[d]^{c_{V'}} \\
\tilde{S}_1\times S_2 \ar[r]_f & \hat{S}_1\times \hat{S}_2 }$$
where $c_{V}=r\times \id$ and $c_{V'}=p_1\times p_2$ where $r:S_1 \rightarrow \tilde{S}_1$ is the contraction of two $E_{11}=r^{-1}(R_1)$ and $E_{12}=r^{-1}(R_2)$ whereas $p_1$ and $p_2$ are the blow-up with exceptional divisor respectively $E_1$ and $E_2$.
Note that the cone $V$ spanned by $E_{11}\times P$ and $E_{12}\times P$ is an extremal subcone because for $a>>0$, $L:=\O((aH-E_{11}-E_{12})\times S_2)$ is a nef line bundle such that $V=\NE(S_1\times S_2)\cap L^{\perp}$. This implies that its image $V'$ is extremal. Again, the construction of $f$ make sense because $c_{V'}$ contracts a curve if and only if $c_{V}$ contracts its preimage and because all the fibers of $c_{V}$ are connected and have at most dimension one.
\vspace{4mm}

\noindent Assume $f(R_1\times Q_1)=f(R_1\times Q_2)$. The fibers $E_{11}\times Q_i$ are mapped to two disjoint curves of the form $\tilde{Q_i}\times E_2$ and then contracted to the same point. Then the fiber $S$ of this point has dimension at least $2$ (exactly $2$ by construction) and contains $\tilde{Q_i}\times E_2$. Recall that $-K_{X_{|_S}}:=D'$ is ample so it intersects $\tilde{Q_i}\times E_2$. $D'$ is then an effective curve that is contracted to a point by $c_{V'}$ so its preimage $D$ intersects $E_{11}\times Q_i$ and is contracted by $c_{V}$. Hence $Q_1=Q_2$. In a similar way we dealt with the other cases and prove that $f$ is injective. By construction, $f$ is also surjective and hence bijective.
\vspace{4mm}

\noindent Again $f$ is a map that is holomorphic outisde two disjoint smooth subvariety of $\tilde{S}_1\times S_2$ whose codimension is $2$. Thus, by Hartogs' Theorem, $f$ is everywhere holomorphic. Then $f$ is an isomorphism but checking the equality of the Euler numbers one obtain $$2+\rho(S_2)=2+\rho(S_1)=\chi(S_1)=\chi(S_2)+1=3+\rho(S_2)$$
and then again a contradiction. Hence the two types of extremal rays cannot be mixed by $\phi$. There are two cases: the first corresponding to the case for which $\forall \phi\in \Aut(S),$ $\phi_*\NE(\pi_i)=\NE(\pi_i)$ and the second where there exists $\phi\in\Aut(X)$ that switches the two cones. By Lemma \ref{LEM:Split}, in the first case $\Aut(S_1\times S_2)=\Aut(S_1)\times \Aut(S_2)$ and $S_1\neq S_2$ whereas, in the second, we have $S_1=S_2$ and $\Aut(S_1\times S_2)=\Aut(S_1)^{\times 2}\ltimes \Z_2$.
\end{proof}

\begin{lem}
Let $S_1$ and $S_2$ be two del Pezzo surfaces with $\rho(S_1)\leq 2$ and $\rho(S_2)\geq 3$. Then $\Aut(S_1\times S_2)=\Aut(S_1)\times \Aut(S_2)$.
\end{lem}

\begin{proof}
There are three cases: $\rho(S_1)=1$ with $S_1=\P^2$ and $\rho(S_1)=2$ with $S_1=\P^1\times \P^1$ or $S_1=dP_8$.
\vspace{4mm}

\noindent If $S_1=\P^2$ and $\phi\in \Aut(X)$, fix a point $s\in S$ and consider the map obtained as composition of the inclusion $\P^2\simeq\P^2\times \left\lbrace s\right\rbrace\subset \P^2\times S_2$, $\phi$ and the projection on $S$. The resulting map $\beta_s$ cannot be a dominant morphism because, in this case, $\P^2$ would have divisors with negative self-intersection\footnote{The pullback $D$ of a $(-1)-$line $E$ for example.}. Moreover its image cannot have dimension greater than $0$; in fact, every surjective map $\P^2\rightarrow C$ induces a surjective map $\P^2\rightarrow \P^1$ but this cannot exist. Hence $\beta_s(\P^2)$ is a point, or equivalently, doesn't depend on $P$. Hence $$\phi(P,s)=(\alpha(P,s),\beta(s))$$
and the same holds true for $\phi^{-1}$ so $\beta\in\Aut(S_2)$ and, by a composition with $\id\times \beta^{-1}$, we can restrict to the case $\beta=\id$. Consider now for a fixed $s\in S_2$ the morphism $\alpha_s:\P^2\rightarrow \P^2$. As before, its image cannot have dimension $1$. If $\Dim(\alpha_s(\P^2))=0$ then $\phi(\P^2\times \left\lbrace s\right\rbrace)\subset Pt\times S_2$, and because $\phi$ is an automorphism, we would obtain an isomorphism between $\P^2$ and a del Pezzo surface of Picard number strictly greater than $1$ which is impossible. Hence $\alpha_s$ is a dominant map. Suppose $\alpha_s(P)=\alpha_s(Q)$. Then $$\phi(P,s)=(\alpha_s(P),s)=(\alpha_s(Q),s)=\phi(Q,s)$$
but $\phi$ is injective so $P=Q$ and $\alpha_s$ is also injective. This shows that $\alpha_s$ is an automorphism for every $s$ and in particular we have a map $f: s\in S_2\mapsto \alpha_s\in PGL(3)=SL(3)/ \Z_3$. Then $f$ lifts to a map from $S_2$ to $SL(3)$ that is affine and then $f$ doesn't depend on $s$. So $\Aut(\P^2\times S)=\Aut(\P^2)\times \Aut(S_2)$.
\vspace{4mm}

\noindent If $S_1=\P^1\times \P^1$ then the extremal rays of $X=S_1\times S_2$ are of the form $[(P_1\times P_2) \times E],[(P_1\times \P^1) \times Q]$ or $[(\P^1\times  P_2) \times Q]$ where $E$ is a $(-1)-$curve on $S_2$. In particular $((P_1\times P_2) \times E)\cdot (K_X)= -1$ whereas
$$((P_1\times \P^1) \times Q)\cdot K_X=((\P^1\times  P_2) \times Q)\cdot K_X=-2.$$
In particular, because extremal rays are permuted by every automorphism and because the intersection numbers are preserved, we have $\phi_*(\NE(\pi_i))=\NE(\pi_i)$ and then $\Aut(S_1\times S_2)=\Aut(S_1)\times \Aut(S_2)$.
\vspace{4mm}

\noindent If $S_1=dP_8$ and $\rho(S_2)\geq 3$ then the extremal rays are of the form $[E \times P_2],[(H-E) \times P_2]$ and $[P_1 \times E_2]$ where $E$ is the only $(-1)-$curve on $S_1$ and $E_2$ is a $(-1)-$curve on $S_2$. In particular $-K_X\cdot ((H-E) \times P_2)=2$ whereas for all the other extremal curves the intersection with $-K_X$ is $1$; hence $\phi_*$ fixes this extremal ray. Assume that $\phi_*([E\times P_2])=([P_1\times E_i])$. Then, denoting $V=\R^+[E\times P_2]$ and $V'=\R^+[P_1\times E_i]$, we obtain the following commutative diagram
$$\xymatrix{\ar @{} [dr] |{\circlearrowleft}
dP_8\times S_2 \ar[d]_{c_{V}} \ar[r]^{\phi} & dP_8\times S_2 \ar[d]^{c_{V'}} \\
\P^2\times S_2 \ar[r]_f & dP_8\times \hat{S}_2 }$$
where $f$ is again an isomorphism. This gives $\chi(S_2)=4$ but $\rho(S_2)\geq 3$ so we have a contradiction ($4=\chi(S_2)\geq 5$). Thus $\NE(S_i)=\phi_*(\NE(S_i))$ and then $\Aut(S_1\times S_2)=\Aut(S_1)\times \Aut(S_2)$.
\end{proof}

\begin{lem}
Let $S_1$ and $S_2$ be two del Pezzo surfaces such that $\rho(S_1),\rho(S_2)\leq 3$. Then:
\begin{itemize}
\item If $S_1\neq S_2$, $\Aut(S_1\times S_2)=\Aut(S_1)\times\Aut(S_2)$;
\item If $S_1=S_2\neq \P^1\times \P^1$, $\Aut(S_1\times S_2)=(\Aut(S_1)\times\Aut(S_2))\ltimes \Z_2$;
\item If $S_1=S_2= \P^1\times \P^1$, $\Aut(S_1\times S_2)=(\Aut(\P^1)^{\times 4})\ltimes S_4.$
\end{itemize}
\end{lem}

\begin{proof}
If $\rho(S_i)\leq 3$, $S_i$ is a smooth toric variety. For a complete simplicial toric variety the sequence
$$1\rightarrow \Aut^0(X)\rightarrow \Aut(X)\rightarrow \frac{\Aut(N,\Delta)}{\Pi S_{\Delta_i}}\rightarrow 1$$
is exact by a result of Cox (see \cite{Cox}). We will see that this extension is a split extension in all our cases and hence $\Aut(X)$ can be seen as a semidirect product of $\Aut^0(X)$ and $\frac{\Aut(N,\Delta)}{\Pi S_{\Delta_i}}$. The proof will be completed analysing the structure of these two groups.
\vspace{4mm}

\noindent We call $\Delta_{S_i}\subset \Z^2=:N_i$ the fan of $S_i$ and denote with $\Delta_{S_i}(1)=\left\lbrace e_0,\dots, e_{r_i}\right\rbrace$ the set of the rays of the fan.
The following table summarizes the rays of the fans we need.
\begin{center}
\begin{tabular}{c|cccccc}
$S$ & $e_1$ & $e_2$ & $e_3$ & $e_4$ & $e_5$ & $e_6$ \\ \hline
$\P^2$ & [1,0] & [0,1] & [-1,-1]\\
$\P^1\times \P^1$ & [1,0] & [0,1] & [-1,0] & [0,-1]\\
$dP_8$ & [1,0] & [0,1] & [-1,0] & [-1,-1]\\
$dP_7$ & [1,0] & [0,1] & [-1,0] & [0,-1] & [-1,-1]\\
$dP_6$ & [1,0] & [0,1] & [-1,0] & [0,-1] & [-1,-1] & [1,1]
\end{tabular}
\end{center}
If $\Delta\subset \Z^4=N$ is the fan of $X$, then $\Delta(1)=(\Delta_{S_1}\times \left\lbrace [0,0]\right\rbrace) \cup (\left\lbrace [0,0]\right\rbrace\times \Delta_{S_2})$. $\Aut(N,\Delta)$ will denote the group of the automorphisms of the lattice $N$ that fixes the fan $\Delta$. By direct computation, we show that
\begin{itemize}
\item If $S_1\neq S_2$, $\Aut(N,\Delta)=\Aut(N_1,\Delta_{S_1})\times \Aut(N_2,\Delta_{S_2})$;
\item If $S_1=S_2\neq \P^1\times \P^1$, $\Aut(N,\Delta)=(\Aut(N_1,\Delta_{S_1})\times \Aut(N_2,\Delta_{S_2}))\ltimes \Z_2$;
\item If $S_1=S_2=\P^1\times \P^1$, $\Aut(N,\Delta)=S_4 \ltimes \Z_2^4$.
\end{itemize}
It is possible to associate a divisor $D_i$ to each $e_i\in \Delta(1)$ and we say than $e_i\sim e_j$ iff $D_i$ and $D_j$ are linearly equivalent. Call $\left\lbrace\Delta_i\right\rbrace$ the partition of $\Delta(1)$ obtained by taking the quotient with respect to $\sim$. Call $S_{\Delta_i}$ the pemutation group over $\Delta_i$. It is easy to see that this partition doesn't mix rays coming from different factors of the product so we can write $S^1_{\Delta_i}$ or $S^2_{\Delta_i}$ to mean a permutation group that acts on the first or on the second factor. Call $H$ the quotient of $\Aut(N,\Delta)$ with respect to $\Pi S_{\Delta_i}=\Pi S^1_{\Delta_i} \times \Pi S^2_{\Delta_i}$. Then
\begin{itemize}
\item If $S_1\neq S_2$, $H=\frac{\Aut(N_1,\Delta_{S_1})}{\Pi S^1_{\Delta_i}}\times \frac{\Aut(N_2,\Delta_{S_2})}{\Pi S^2_{\Delta_i}}$;
\item If $S_1=S_2\neq \P^1\times \P^1$, $H=\left(\frac{\Aut(N_1,\Delta_{S_1})}{\Pi S^1_{\Delta_i}}\times \frac{\Aut(N_2,\Delta_{S_2})}{\Pi S^2_{\Delta_i}}\right)\ltimes \Z_2$;
\item If $S_1=S_2=\P^1\times \P^1$, $H=\frac{S_4 \ltimes Z_2^4}{\Z_2^4}\simeq S_4$.
\end{itemize}
Here a small summary of these groups.
\begin{center}
\begin{tabular}{c|ccc}
S &  $\Aut(N_{S},\Delta_{S})$  &  $\prod S_{\Delta_i}$ & $\Aut(N_{S},\Delta_{S})/\prod S_{\Delta_i}$\\ \hline
$\P^2$ & $\Sym(e_1,e_2,e_3)$	& $\Sym(e_1,e_2,e_3)$ & $\id$ \\
$\P^1\times \P^1$ 	& $<(13),(1234)>$	& $<(13),(24)>$ & $\Z_2$ \\
$dP_8$ & $<(24)>$ & $<(24)>$ & $\id$  \\
$dP_7$ & $<(12)(34)>$ & $\id$ & $\Z_2$ \\
$dP_6$ &  $\Sym(e_1,e_2,e_5)\times <-\id>$ & $\id$ & $S_3\times \Z_2$
\end{tabular}
\end{center}
To see that the sequence splits, consider, for example, the case $X=dP_8\times dP_7$ for which $H=\id\times \Z_2=<\sigma>$. This group is generated by the automorphism of the fan of $dP_7$ that switches the rays associated to the two exceptional divisors of $dP_7$, thus a section of $\Aut(X)\rightarrow H$ is given by
$\sigma \mapsto A$ where $A$ is an automorphism of $\P^2$ that switches the two points that are blown-up to obtain $dP_7$. All the other cases can be described in a similar way.
\vspace{4mm}

\noindent $\Aut^0(X)$ is the connected component of the identity in $\Aut(X)$ and now we will show that  $\Aut^0(X)=\Aut^0(S_1)\times \Aut^0(S_2)$. By a result of Cox (see again \cite{Cox}) $$\Aut^0(X)\simeq \frac{\Aut_g(S)}{Hom_{\Z}(\Pic(X),\C^*)}$$ where $\Aut_g(S)$ is the group of the automorphisms of the  homogeneous coordinate ring $S$ of $X$, regarded as graded $\C-$algebra. This group is spanned by $(\C^*)^{|\Delta(1)|}= (\C^*)^{|\Delta_{S_1}(1)|+|\Delta_{S_2}(1)|}$ and by the elements $y_m(\lambda)$ where $\lambda\in \C$ and $m\in R(N,\Delta)$ (the elements of $R(N,\Delta)$ are the roots of $\Aut(X)$). We show that each $y_m(\lambda)$ can be written in a unique way as the product of $f_i\in \Aut_g(R_i)$ where $R_i$ is the coordinate ring of $S_i$. This shows that $\Aut_g(S)\simeq \Aut_g(R_1)\times\Aut_g(R_2)$. The group $Hom_{\Z}(\Pic(X),\C^*)$ splits as $Hom_{\Z}(\Pic(S_1),\C^*)\times Hom_{\Z}(\Pic(S_2),\C^*)$ because $\Pic(X)=\Pic(S_1)\oplus\Pic(S_2)$. Then, the quotient can be viewed as a product of the quotient giving
$$\Aut^0(X)=\Aut^0(S_1)\times \Aut^0(S_2).$$
The claim follows from the combination of the facts above. For example, consider again the case $X=dP_8\times dP_7$. Since $\Aut(dP_8)$ is connected, we have $\Aut^0(X)=\Aut(dP_8)\times K$, where $$K\simeq \left\langle\begin{bmatrix}1 & 0 & * \\ 0 & * & * \\ 0 & 0 & *
\end{bmatrix}\right\rangle.$$ Since $H=\id\times \Z_2$,  we obtain $$\Aut(X)\simeq (\Aut(dP_8)\times K)\ltimes (\id\times \Z_2)=$$
$$\Aut(dP_8)\times (K\ltimes \Z_2)=\Aut(dP_8)\times \Aut(dP_7).$$
\end{proof}

\vspace{4mm}
\noindent Combining all these results, we obtain
\begin{thm}
\label{THM:AUTdPs}
Let $S_1$ and $S_2$ be two del Pezzo surfaces. Then
\begin{itemize}
\item If $S_1\neq S_2$, $\Aut(S_1\times S_2)=\Aut(S_1)\times\Aut(S_2)$;
\item If $S_1=S_2\neq \P^1\times \P^1$, $\Aut(S^{\times 2})=\Aut(S)^{\times 2}\ltimes \Z_2$;
\item If $S_1=S_2= \P^1\times \P^1$, $\Aut((\P^1)^{\times 4})=\Aut(\P^1)^{\times 4}\ltimes S_4.$
\end{itemize}
\end{thm}


\section{List of the Threefolds Obtained}\label{list}

In the previous sections we constructed examples of quotients of Calabi-Yau threefolds $Y$ embedded in $S_1\times S_2$ by groups that are of maximal order in the sense that a group $H\leq \Aut(S_1\times S_2)$ such that the restriction to $Y$ gives a free action, cannot have greater order than the ones used. If $Y$ is a Calabi-Yau threefold and $G$ is a group acting freely on $Y$ the same holds true each $H\leq G$. Moreover $Y/H\rightarrow Y/G$ is an \'etale covering. In the following table we summarize all the quotients analyzed and all the \'etale coverings obtained by taking quotient with respect to subgroups. Also the known examples are shown.
The column $m(|G|)/M$ represents the ratio of the maximal order of the existing group action freely on $Y$ and the estimated ($M=M(S_1,S_2)$). In the column $\Pi_1(Y/G)$ the fundamental group of the quotient is written. When for two isomorphic subgroups $H_1$ and $H_2$ of $G$ we obtain $h^{11}(Y/H_1)=h^{11}(Y/H_2)$ and  $h^{12}(Y/H_1)=h^{12}(Y/H_2)$ we represent them in the table in one row indicating that multiple subgroups give the same result by their number between round brackets. For example, taking $S_1=S_2=\P^2$ and $G\simeq \Z_3\oplus \Z_3$ there are $4$ subgroups of order $3$ and each of them gives a manifold with Hodge numbers $(2,29)$. In the table this is summarized by writing $\Z_3 (4)$ in the column of $\Pi_1(Y/H)$. In the last column a $"Y"$ means that the height obtained for the quotient threefold is the least possible, a $"N"$ means the opposite and a $"?"$ means that we don't know if this is the case or not. The pairs $(S_1,S_2)$ for which $M(S_1,S_2)=1$ are omitted.
\vspace{4mm}

\begin{center}
\begin{tabular}{|c|c|c|c|c|c|c|c|c|c|} \hline
$S_1$&$S_2$&$\max(|G|)/M$&$|G|$&$\Pi_1(Y/H)$&$h^{11}$&$h^{12}$&$\h$ & $\min$? \\ \hline\hline
\multirow{3}{*}{$\P^2$} & \multirow{3}{*}{$\P^2$} & \multirow{3}{*}{$9/ 9$} & 9 & $\Z_3\oplus \Z_3$ & 2 & 11 & 13 & Y \\ \cline{4-9}
 & & & 3 & $\Z_3 (4)$ & 2 & 29 & 31 & N \\ \cline{4-9}
 & & & 1 & $\left\lbrace\id\right\rbrace$ & 2 & 83 & 85 & N \\ \hline

\multirow{2}{*}{$\P^2$} & \multirow{2}{*}{$dP_6$} & \multirow{2}{*}{$3/ 3$} & 3 & $\Z_3$ & 3 & 21 & 24 & Y \\ \cline{4-9}
 & & & 1 & $\left\lbrace\id\right\rbrace$ & 5 & 59 & 64 & N \\ \hline

\multirow{2}{*}{$\P^2$} & \multirow{2}{*}{$dP_3$} & \multirow{2}{*}{$3/ 3$} & 3 & $\Z_3$ & 4 & 13 & 17 & Y \\ \cline{4-9}
 & & & 1 & $\left\lbrace\id\right\rbrace$ & 8 & 35 & 43 & N \\ \hline

\multirow{7}{*}{$\P^1\times \P^1$} & \multirow{7}{*}{$\P^1\times \P^1$} & \multirow{7}{*}{$16/ 16$} & 16 & $\Z_8\oplus\Z_2$ & 1 & 5 & 6 & Y \\ \cline{4-9}
 & & & 8 & $\Z_4\oplus\Z_2$ & 2 & 10 & 12 & N \\ \cline{4-9}
 & & & 8 & $\Z_8 (2)$ & 1 & 9 & 10 & N \\ \cline{4-9}
 & & & 4 & $\Z_2\oplus\Z_2$ & 4 & 20 & 24 & N \\ \cline{4-9}
 & & & 4 & $\Z_4 (2)$ & 2 & 18 & 20 & N \\ \cline{4-9}
 & & & 2 & $\Z_2 (3)$ & 4 & 36 & 40 & N \\ \cline{4-9}
 & & & 1 & $\left\lbrace\id\right\rbrace$ & 4 & 68 & 72 & N \\ \hline

\multirow{2}{*}{$\P^1\times \P^1$} & \multirow{2}{*}{$dP_6$} & \multirow{2}{*}{$2/ 2$} & 2 & $\Z_2$ & 5 & 29 & 34 & Y \\ \cline{4-9}
 & & & 1 & $\left\lbrace\id\right\rbrace$ & 6 & 54 & 60 & N \\ \hline

\multirow{6}{*}{$dP_6$} & \multirow{6}{*}{$dP_6$} & \multirow{6}{*}{$12/ 12$} & 12 & $\Z_{12}$ & 1 & 4 & 5 & Y \\ \cline{4-9}
 & & & 6 & $\Z_6$ & 2 & 8 & 10 & N \\ \cline{4-9}
 & & & 4 & $\Z_4$ & 3 & 12 & 15 & N \\ \cline{4-9}
 & & & 3 & $\Z_3$ & 4 & 16 & 20 & N \\ \cline{4-9}
 & & & 2 & $\Z_2$ & 6 & 24 & 30 & N \\ \cline{4-9}
 & & & 1 & $\left\lbrace\id\right\rbrace$ & 8 & 44 & 52 & N \\ \hline

\multirow{6}{*}{$dP_6$} & \multirow{6}{*}{$dP_6$} & \multirow{6}{*}{$12/ 12$} & 12 & $\Dic_{3}$ & 1 & 4 & 5 & Y \\ \cline{4-9}
 & & & 6 & $\Z_6$ & 2 & 8 & 10 & N \\ \cline{4-9}
 & & & 4 & $\Z_4 (3)$ & 3 & 12 & 15 & N \\ \cline{4-9}
 & & & 3 & $\Z_3$ & 4 & 16 & 20 & N \\ \cline{4-9}
 & & & 2 & $\Z_2$ & 6 & 24 & 30 & N \\ \cline{4-9}
 & & & 1 & $\left\lbrace\id\right\rbrace$ & 8 & 44 & 52 & N \\ \hline

\multirow{2}{*}{$dP_6$} & \multirow{2}{*}{$dP_4$} & \multirow{2}{*}{$2/ 2$} & 2 & $\Z_2$ & 7 & 19 & 26 & ? \\ \cline{4-9}
 & & & 1 & $\left\lbrace\id\right\rbrace$ & 10 & 34 & 44 & N \\ \hline

\multirow{2}{*}{$dP_6$} & \multirow{2}{*}{$dP_3$} & \multirow{2}{*}{$3/ 9$} & 3 & $\Z_3$ & 5 & 11 & 16 & Y \\ \cline{4-9}
 & & & 1 & $\left\lbrace\id\right\rbrace$ & 11 & 29 & 40 & N \\ \hline

\multirow{2}{*}{$dP_5$} & \multirow{2}{*}{$dP_5$} & \multirow{2}{*}{$5/ 5$} & 5 & $\Z_5$ & 2 & 7 & 9 & Y \\ \cline{4-9}
 & & & 1 & $\left\lbrace\id\right\rbrace$ & 10 & 35 & 45 & N \\ \hline

\multirow{5}{*}{$dP_4$} & \multirow{5}{*}{$dP_4$} & \multirow{5}{*}{$8/ 8$} & 8 & $\Z_4\oplus\Z_2$ & 3 & 5 & 8 & ? \\ \cline{4-9}
 & & & 4 & $\Z_2\oplus\Z_2$ & 6 & 10 & 16 & N \\ \cline{4-9}
 & & & 4 & $\Z_4 (2)$ & 4 & 8 & 12 & N \\ \cline{4-9}
 & & & 2 & $\Z_2 (3)$ & 8 & 16 & 24 & N \\ \cline{4-9}
 & & & 1 & $\left\lbrace\id\right\rbrace$ & 12 & 28 & 40 & N \\ \hline

\multirow{2}{*}{$dP_3$} & \multirow{2}{*}{$dP_3$} & \multirow{2}{*}{$3/ 3$} & 3 & $\Z_3$ & 6 & 9 & 15 & Y \\ \cline{4-9}
 & & & 1 & $\left\lbrace\id\right\rbrace$ & 14 & 23 & 37 & N \\ \hline

\multirow{3}{*}{$\P^1\times \P^1$} & \multirow{3}{*}{$dP_4$} & \multirow{3}{*}{$4/ 4$} & 4 & $\Z_2\oplus\Z_2$ & 5 & 13 & 18 & ? \\ \cline{4-9}
 & & & 2 & $\Z_2 (3)$ & 6 & 22 & 28 & N \\ \cline{4-9}
 & & & 1 & $\left\lbrace\id\right\rbrace$ & 8 & 40 & 48 & N \\ \hline
\end{tabular}
\end{center}

\clearpage
\addcontentsline{toc}{section}{Bibliography}
\bibliography{calabiyau}

\end{document}